\documentclass{imsart}

\usepackage{amsfonts}
\usepackage{amssymb}
\usepackage{graphicx}
\usepackage{amsmath}
\usepackage{color}

\newtheorem{theorem}{Theorem}

\newtheorem{condition}{Condition}

\newtheorem{definition}{Definition}

\newtheorem{lemma}{Lemma}

\newtheorem{remark}{Remark}

\newenvironment{proof}[1][Proof]{\noindent\textbf{#1.} }{\ \rule{0.5em}{0.5em}}

\newcommand{\cc}{\mathbb C}
\newcommand{\rr}{\mathbb R}
\newcommand{\ee}{\mathbb E}
\newcommand{\pp}{\mathbb P}
\newcommand{\mm}{\mathbb M}
\newcommand{\hh}{\mathbb H}
\newcommand{\id}{\mathbb I}

\definecolor{jens}{rgb}{0,0.4,0.8}

\begin{document}

\begin{frontmatter}

\title{Uncertainty Quantification for Matrix Compressed Sensing and Quantum Tomography Problems}
\runtitle{Uncertainty Quantification for Matrix CS}

\begin{aug}
\author{\fnms{Alexandra} \snm{Carpentier}\thanksref{m1}\ead[label=e1]{carpentier@maths.uni-potsdam.de}},
\author{\fnms{Jens} \snm{Eisert}\thanksref{m2}\ead[label=e2]{jense@physik.fu-berlin.de}}, \\
\author{\fnms{David} \snm{Gross}\thanksref{m3}\ead[label=e3]{david.gross@thp.uni-koeln.de}}
\author{\fnms{Richard} \snm{Nickl}\thanksref{m4}\ead[label=e4]{r.nickl@statslab.cam.ac.uk}}
%\ead[label=u1,url]{http://www.foo.com}}

\runauthor{A.~Carpentier, J.~Eisert, D.~Gross and R.~Nickl}

\affiliation{Universit\"at Potsdam\thanksmark{m1}, Freie Universit\"at Berlin\thanksmark{m2}, \\ Universit\"at zu K\"oln \thanksmark{m3}, and University of Cambridge\thanksmark{m4}}

\address{
\printead{e1}\\
\phantom{E-mail:\ }\printead*{e2}}

\address{
\printead{e3}\\
\printead{e4}}
\end{aug}

\maketitle

\begin{abstract}
We construct minimax optimal non-asymptotic confidence sets for low rank matrix recovery algorithms. These are employed to devise sequential sampling procedures that guarantee recovery of the true matrix in Frobenius norm after a data-driven stopping time $\hat n$ for the number of measurements that have to be taken. With high probability, this stopping time is minimax optimal. We detail applications to quantum tomography problems where measurements arise from Pauli observables. We also give a theoretical construction of a confidence set for the density matrix of a quantum state that has optimal diameter in nuclear norm. The non-asymptotic properties of our confidence sets are investigated in a simulation study.

\end{abstract}

\end{frontmatter}

\section{Introduction}

Consider the high-dimensional matrix trace regression model
\begin{equation} \label{model}
Y_i = tr (X^i \theta) + \varepsilon_i,\quad i=1, \dots, n,
\end{equation}
where the $\varepsilon_i$'s are random noise variables, independent of the random design matrices $X^i$, and where the matrix $\theta$ is the object of inferential interest. We denote the law of $(Y,X)$ given $\theta$ by $\mathbb P_\theta$. To reflect the structure of the main application we have in mind -- quantum tomography, introduced in detail below -- we assume that $X^i$ and $\theta$ are both $d \times d$ square matrices, and study the case where the number $n$ of measurements taken may be smaller than the effective parameter dimension $d^2$. Recovery of $\theta$ in such situations is still possible by compressed sensing techniques \cite{CP11,KLT11,K11,Recht}, under two main structural assumptions on the model: 1) the matrix $\theta$ is of low rank and 2) the measurement matrices $X^i$ satisfy the restricted isometry (or a related coherence) property. In this case recovery of a rank $k$ matrix $\theta$ is possible in Frobenius distance $\|\cdot\|_F$ by, e.g., the Matrix Lasso $\hat \theta_{n}$: for any $\epsilon>0$ and with high $\mathbb P_\theta$-probability, $$\|\hat \theta_{n} - \theta\|_F <\epsilon~~~\text{as soon as }~ n \gtrsim kd \overline{\log} d,$$ where $\overline \log d = \log^\eta d, \eta>0,$ is the so-called `polylog' function. The design used in quantum tomography is such that the $X^i$ are randomly drawn from a basis $\{E_1, \dots, E_{d^2}\}$ of the space of $d \times d$ matrices, and one samples fewer than all $d^2$ basis coefficients $tr(E_i \theta)$ without losing recovery guarantees for low rank matrices $\theta$. In experimental settings (e.g., \cite{Blatt}), $d=2^N$ where $N$ is a possibly large number of particles, but $\theta$ will represent an approximately pure quantum state, motivating the low rank hypothesis and explaining the interest of quantum information theorists in dimension reduction methods (see the appendix and \cite{GLFBE10, G11, L11, FGLE12, Guta, Riofrio}).

%\smallskip

In practice the implementation of the compressed sensing paradigm requires a way to decide how many measurements $n$ should be taken. The preceding theoretical bound $n \gtrsim k d \overline \log d$ is not useful for this because it may involve unspecified constants, but also, more importantly, because the rank $k$ of $\theta$ is typically not known. Instead one can try to find a \textit{data driven stopping rule} $\hat n$ that guarantees that recovery with precision $\epsilon$ occurs after $\hat n$ measurements, with high probability. In the quantum tomography context such stopping rules are called `certificates' (see Section IV in \cite{FGLE12}), as they certify the reconstruction of the true quantum state $\theta$. It is not difficult to see, and will be made precise below, that the construction of such stopping rules is intimately connected to the construction of a (sequential) confidence region for the unknown parameter $\theta$, and due to its importance in applications this topic has received considerable attention recently by physicists, see \cite{Christandl,Robin,temme2015quantum,audenaert2009quantum, FGLE12, shang2013optimal}. None of the previous constructions has succeeded, however, in constructing an \textit{optimal} stopping rule for which $\hat n \approx kd \overline \log d$ holds with high probability. 

%\smallskip

The main contribution of the present paper is to construct optimal non-asymptotic Frobenius norm confidence regions for low rank parameters in the model (\ref{model}), and to use them to devise optimal sequential data driven stopping rules $\hat n$ (`certificates') for the measurement process. That such procedures exist may at first look surprising in view of negative results in the `sparse' compressed sensing setting in \cite{NvdG13}, but our results reveal the more favourable information-theoretic structure of the matrix model. While our techniques are based on unbiased risk estimation ideas that were first used in nonparametric statistics (see \cite{L89, RV06}, and also \cite{BN13}) and that apply in a general setting, we lay out the details for a basic (sub-) Gaussian design and noise model, as well as for the Pauli observation scheme relevant in quantum tomography (see \cite{FGLE12} and Condition \ref{design}b) below). We shall also address the more difficult question of constructing confidence regions for a quantum state matrix in the stronger nuclear norm. Relationships between our findings and the recent literature on confidence regions for high-dimensional statistical parameters are discussed at the end. We also investigate the performance of our procedures in basic simulation study.

\section{The framework of matrix compressed sensing}

\subsection{Notation}

Denote by $\mm_d(\mathbb K)$ the space of $d \times d$ matrices with entries in $\mathbb K = \cc$ or $\mathbb K=\rr$. We write $\|\cdot\|_F$ for the usual Frobenius norm on $\mm_d(\mathbb K)$ arising from the inner product $tr (A^TB) = \langle A, B \rangle_F$. Moreover let $\hh_d(\mathbb K)$ be the set of all Hermitian matrices (equal to the set of all symmetric $d \times d$ matrices when $\mathbb K = \mathbb R$). The norm symbol $\|\cdot\|$ denotes the standard Euclidean norm on $\cc^n$ arising from the Euclidean inner product $\langle \cdot, \cdot \rangle$.

%\smallskip

We denote the usual operator norm on $\mm_d(\mathbb K)$ by $\|\cdot\|_{op}$. For $M\in \mm_d(\mathbb K)$ let $(\lambda_k^2: k =1, \dots, d)$ be the eigenvalues of $M^TM$ (which are all real-valued and positive). The $l_1$-Schatten, \textit{trace}, or \textit{nuclear} norm of $M$ is defined as
$$\|M\|_{S_1} = \sum_{j\leq d} |\lambda_j|.$$
Note that for any matrix $M$ of rank $1\le r \le d$,
\begin{equation} \label{l1l2}
\|M\|_{F} \leq \|M\|_{S_1}\leq \sqrt{r}\|M\|_{F}.
\end{equation}

%\smallskip

We will consider parameter subspaces of $\hh_d(\mathbb K)$ described by low
rank constraints on $\theta$, and denote by $R(k)$ the space of all
Hermitian $d \times d$ matrices that have rank at most $k$, $k \le d$.
In quantum tomography applications, we may assume an additional `shape
constraint', namely that $\theta$ is a density matrix of a 
quantum state, and hence contained in \emph{state space} $$\Theta_{+} =
\{\theta \in \hh_d(\mathbb K): tr(\theta)=1, \theta \succeq 0\},$$ where
$\theta \succeq 0$ means that $\theta$ is positive semi-definite. In
fact, in most situations, we will only require the bound
$\|\theta\|_{S_1} \le 1$ which holds for any $\theta$ in
$\Theta_+$.

%\smallskip

\subsection{Sensing matrices}

We now specify assumptions on the design matrices $X^i$ used in our observation model (\ref{model}). When $\theta$ has real-valued entries we shall restrict to design matrices
$X^i$ with real-valued entries too, and for general $\theta \in
\hh_d(\cc)$ we shall assume $X^i \in \hh_d(\cc)$. This
way, in either case, the measurements $tr(X_i \theta)$'s and hence the
$Y_i$'s are all real-valued. Note that in Part a) below the design matrices are not Hermitian but our results can easily be generalised to symmetrised sub-Gaussian ensembles (as those considered in ref. \cite{K11}). Part b) corresponds to the quantum tomography measurement model used in \cite{G11, L11, FGLE12, GLFBE10} -- we refer to the appendix for a detailed derivation.

%More concretely the sensing matrices $X^i$ that we shall consider are described in the following assumption, which encompasses both a prototypical compressed sensing setting -- where we can think of the matrices $X^i$ as i.i.d.~draws from a Gaussian ensemble $(X_{m,k}) \sim^{iid} N(0,1)$  -- as well as the `random sampling from a basis of $\mm_d(\cc)$' scenario.
%relevant in quantum tomography problems, where frequently the Pauli
%basis is used (see, e.g., refs. \cite{G11, L11, FGLE12, GLFBE10}).  The systematic study of the latter has been initiated by quantum physicists \cite{G11, L11}, as it contains, in particular, the case of Pauli basis measurements \cite{GLFBE10, FGLE12} frequently employed in quantum tomography problems.

\begin{condition} \label{design} 
\begin{itemize}
\item[a)] {\bf $\theta \in \hh_d(\rr)$, `isotropic' sub-Gaussian design:} The random variables $(X^i_{m,k})$, $1 \le m,k \le d, i =1, \dots, n,$ generating the entries of the random matrix $X^i$ are i.i.d.~distributed across all indices $i,m,k$ with mean zero and unit variance. Moreover, for every $\theta \in \mm_d(\rr)$ such that $\|\theta\|_F \le 1$ the real random variables $Z_i=tr(X^i\theta)$ are sub-Gaussian:  for some fixed constants $\tau_1, \tau_2>0$ independent of $\theta$, 
$$\ee e^{\lambda Z_i} \le \tau_1 e^{\lambda^2 \tau_2^2} ~\forall \lambda \in \rr.$$

\item[b)]  {\bf $\theta \in \hh_d(\cc)$, random sampling from a basis (`Pauli design'):
} Let $\{E_1, \dots, E_{d^2}\} \subset \hh_d(\cc)$ be a basis of $\mm_d(\cc)$
that is orthonormal for the scalar product $\langle \cdot, \cdot \rangle_F$ and such that the operator norms satisfy, for all $i=1, \dots, d^2$, $$\|E_i\|_{op} \le \frac{K}{\sqrt d},$$ for some $K>0$. [In the Pauli basis case we have $K=1$.] Assume the $X^i$, $ i=1, \dots, n,$ are draws from the finite family $\mathcal E = \{d  E_i: i=1, \dots, d^2\}$ sampled uniformly at random. 
%\item[c)]  {\bf \je{$\theta \in \hh_d(\cc)$, random Pauli bases measurements: }}
\end{itemize}
\end{condition}

The above examples all obey the \textit{matrix restricted isometry property}, that we describe now. Note first that if $\mathcal X: \rr^{d\times d} \to \rr^n $ is the linear `sampling' operator 
\begin{equation} \label{sampling}
\mathcal X: \theta \mapsto \mathcal X \theta = (tr(X^1 \theta) , \dots, tr(X^n \theta))^T,
\end{equation} 
so that we can write the model equation (\ref{model}) as $Y=\mathcal X \theta + \varepsilon$, then in the above examples we have the `expected isometry' $$\ee \frac{1}{n}\|\mathcal X \theta\|^2  = \|\theta\|_F^2.$$ Indeed, in the isotropic design case we have 
\begin{equation} \label{isoexp}
\frac{1}{n}\ee\|\mathcal X \theta\|^2= \frac{1}{n}\sum_{i=1}^n \ee \left(\sum_m \sum_k X^{i}_{m,k} \theta_{m,k} \right)^2 = \sum_m \sum_k \ee X^2_{m,k} \theta_{m,k}^2 = \|\theta\|_F^2,
\end{equation}
and in the `basis case' we have, from Parseval's identity and since the $X^i$'s are sampled uniformly at random from the basis,
\begin{equation}\label{paulexp}
\frac{1}{n}\ee \|\mathcal X \theta\|^2= \frac{d^2}{n} \sum_{i=1}^n \sum_{j=1}^{d^2} \Pr(X^i=E_j) |\langle E_j, \theta \rangle_F|^2 = \|\theta \|_F^2.
\end{equation}
The restricted isometry property (RIP) requires that this `expected isometry' holds, up to constants and with probability $\ge 1-\delta$, for a given realisation of the sampling operator, and for all $d \times d$ matrices $\theta$ of rank at most $k$:
\begin{equation} \label{RIP}
\sup_{\theta \in R(k)}\left|\frac{\frac{1}{n}\|\mathcal X \theta\|^2 - \|\theta\|_F^2}{\|\theta\|_F^2} \right| \le \tau_n(k),
\end{equation}
where $\tau_n(k)$ are some constants that may depend, among other things, on the rank $k$ and the `exceptional probability' $\delta$. For the above examples of isotropic and Pauli basis design inequality (\ref{RIP}) can be shown to hold with 
\begin{equation} \label{tau}
\tau^2_n(k) = c^2 \frac{k d \cdot \overline{\log} d}{n},
\end{equation}
where $c=c(\delta)=O(1/\delta^2)$ as $\delta \to 0$ is a fixed constant.  See refs. \cite{CP11, L11} for these results.

\subsection{Gaussian and Bernoulli errors, and Pauli observables} \label{bernoulli}

We still have to specify the distribution of the errors $\varepsilon_i$ in the model (\ref{model}). In the quantum tomography setting of Condition \ref{design}b), if we fix an element $E_i \in \mathcal E$ for the moment, then as detailed in the appendix the observations $Y_i= d tr(E_i\theta) + \varepsilon_i$ are themselves an average of repeated samples from a Bernoulli random variable $B_i$ taking values $\{1,-1\}$ with probabilities given by
$$\pp(B_i=1)=\frac{1+\sqrt d tr(E_i \theta)}{2}.$$ More precisely, $$Y_i = \frac{\sqrt d}{T}\sum_{j=1}^T B_{i,j} = d \cdot tr(E_i \theta) + \varepsilon_i$$ where $$\varepsilon_i = \frac{\sqrt d}{T}\sum_{j=1}^T (B_{i,j}-\ee B_{i,j})$$ is the effective error arising from the measurement procedure making use of $T$ preparations to estimate each quantum mechanical expectation value. We could work with this Bernoulli error model directly, but since the $\varepsilon_i$'s are themselves sums of independent random variables, an approximate Gaussian error model will be appropriate, too. Note further that 
\begin{equation} \label{bernbds}
|\varepsilon_i| \le 2\sqrt d,~\ee \varepsilon_i^2 \le \frac{d}{T} {\rm Var}(B_{i,1}) \le \frac{d}{T}
\end{equation}
so the variances $E\varepsilon_i^2 = \sigma^2$ are bounded by $d/T$. A natural assumption is then
\begin{condition}\label{gausse}
The $\varepsilon_i, i=1, \dots, n,$ are i.i.d.~$N(0, \sigma^2)$  where $\sigma^2 \le v$ for some known constant $v$.
\end{condition}
This unifies the exposition for both designs considered in Condition \ref{design}, but we note that our proofs are valid in the exact Bernoulli error model as well, see Remark \ref{bernprf} below.

\subsection{Minimax estimation under the RIP}

Assuming the matrix RIP from (\ref{RIP}) to hold and Gaussian noise $\varepsilon \sim N(0,\sigma^2I_n)$, one can show that the minimax risk for recovering a Hermitian rank $k$ matrix is
\begin{equation} \label{mrisk}
\inf_{\hat \theta} \sup_{\theta \in R(k)} \ee _\theta \|\hat \theta- \theta\|^2_F \simeq \sigma^2 \frac{d k}{n},
\end{equation}
where $\simeq$ denotes two-sided inequality up to universal constants. For the upper bound one can use the nuclear norm minimisation procedure or matrix Dantzig selector from Cand\`es and Plan \cite{CP11} (see also \cite{L11} for the case of Pauli-design), and needs $n$ to be large enough so that the matrix RIP holds with $\tau_n(k)<c_0$ where $c_0$ is a small enough numerical constant. Such an estimator $\tilde \theta$ then satisfies, for every $\theta \in R(k)$ and those $n \in \mathbb N$ for which $\tau_n(k)<c_0$,
\begin{equation} \label{cplan}
\|\tilde \theta- \theta\|^2_F \leq D(\delta) \sigma^2 \frac{k d}{n},
\end{equation}
with probability $\ge 1-2\delta$, and with the constant $D(\delta)$ depending on $\delta$ and also on $c_0$ (suppressed in the notation). Note that the results in \cite{CP11} use a different scaling in sample size in their Theorem 2.4, but eq.\ (II.7) in that reference explains that this is just a matter of renormalisation. The same bound holds for the Bernoulli noise model from Subsection \ref{bernoulli}, see \cite{FGLE12}.

\section{Main results} \label{UQlr}

We now turn to the problem of quantifying the uncertainty of estimators $\tilde \theta$ that satisfy the risk bound (\ref{cplan}). In fact the procedures we construct could be used for any estimator of $\theta$, but the conclusions are most interesting when used for minimax optimal estimators $\tilde \theta$. 

\subsection{Confidence sets and sequential sampling protocols}

From a statistical point of view the problem at hand is the one of constructing a confidence set for $\theta$: a data-driven subset $C_n$ of $\mm_d(\cc)$ that is `centred' at $\tilde \theta$, that satisfies $$\pp_\theta (\theta \in C_n) \ge 1- \alpha,~~~0<\alpha<1,$$ for a chosen `coverage' or significance level $1-\alpha$, and such that the Frobenius norm diameter $|C_n|_F$ reflects the accuracy of estimation, that is, it satisfies, with high probability, $$|C_n|^2_F \approx \|\tilde \theta - \theta\|_F^2.$$ In particular such a confidence set provides, through its diameter $|C_n|_F$, a data-driven estimate of how well the algorithm has recovered the true matrix $\theta$ in Frobenius-norm loss, and in this sense provides a quantification of the uncertainty in the estimate. 

In an experimental situation confidence sets $(C_n: n \in \mathbb N)$ can be used to decide sequentially whether more measurements should be taken (to improve the recovery rate), or whether a satisfactory performance has been reached. Concretely, for given $n$ we check if $|C_n|_F \le \epsilon$, and continue to take further measurements if not. Assuming $\tilde \theta$ satisfies the minimax optimal risk bound $dk/n$ from (\ref{cplan}), we expect to need, ignoring constants, $$ \frac{d k}{n}<  \epsilon^2 \text{ and hence at least } n > \frac{d k}{\epsilon^2} $$ measurements. Note that we also need the RIP to hold with $\tau_n(k)$ from (\ref{tau}) less than a small constant $c_0$, which requires the same number of measurements, increased by a further poly-log factor of $d$ (and independently of $\sigma$). The goal is then to prove that a sequential procedure based on $C_n$ does \textit{not} require more than approximately $$n>\frac{dk \overline{\log} d}{\epsilon^2}$$ samples (with high probability). This is made precise in the following definition, where we recall that $R(k)$ denotes the set of $d \times d$ Hermitian matrices of rank at most $k \le d$.

\begin{definition} \label{alg} Let $\epsilon>0, \delta>0$ be given constants. An algorithm $\mathcal A$ returning a $d \times d$ matrix $\hat \theta$ after $\hat n \in \mathbb N$ measurements in model (\ref{model}) is called an $(\epsilon, \delta)$ - adaptive sampling procedure if, with $\pp_\theta$-probability greater than $1-\delta$, the following properties hold for every $\theta \in R(k)$ and every $1 \le k \le d$: 
\begin{equation} \label{recov}
\|\hat \theta - \theta \|_{F} \le \epsilon,
\end{equation}
and, for some positive constants $C(\delta), \gamma,$ the stopping time $\hat n$ satisfies 
\begin{equation} \label{messungen}
\hat n \le C(\delta) \frac{k d (\log d)^\gamma}{\epsilon^2} .
\end{equation}
\end{definition}

Such an algorithm provides recovery at given accuracy level $\epsilon$ with $\hat n$ measurements of minimax optimal order of magnitude (up to a poly-log factor), and with probability greater than $1-\delta$. The sampling algorithm is adaptive since it does not require the knowledge of $k$, and since the number of measurements required depends only on $k$ and not on the `worst case' rank $d$.  

%\smallskip

Our first main result is the following theorem, whose proof relies on the construction of non-asymptotic confidence sets $C_n$ for $\theta$ at any sample size $n$, given in the next subsection. 

\begin{theorem} \label{adsamp}
Consider observations in the model (\ref{model}) under Conditions \ref{design}b) and \ref{gausse}, and where $\theta \in \Theta_+$. Then an adaptive sampling algorithm in the sense of Definition \ref{alg} exists for any $\epsilon, \delta>0$. 
\end{theorem}

The result above holds for isotropic design from Condition \ref{design}a) too, without the constraint $\theta \in \Theta_+$, see Remark \ref{mod} below. For Pauli design  the assumption $\theta \in \Theta_+$ (instead of just $\theta \in \mathbb M_d(\mathbb K)$) is, however, necessary: Else the example of $\theta =0 $ or $\theta=E_i$ -- where $E_i$ is an arbitrary element of the Pauli basis -- demonstrates that the number of measurements has to be at least of order $d^2$: otherwise with positive probability $E_i$ is not drawn at a fixed sample size. On this event both the measurements and $\hat \theta$ coincide under the laws $\pp_0$ and $\pp_{E_i}$, so we cannot have $\|\hat \theta-0\|_F < \epsilon$ AND $\|\hat \theta - E_i\|_F<\epsilon$ simultaneously for every $\epsilon>0$, disproving existence of an adaptive sampling algorithm. In fact, the crucial condition for Theorem \ref{adsamp} to work is that the nuclear norms $\|\theta\|_{S_1}$ are bounded by an absolute constant (here $=1$), which is violated by $\|E_i\|_{S_1} = \sqrt d$.

\subsection{Frobenius norm confidence sets based on unbiased risk estimation}

\subsubsection{An optimal confidence region for $n \le d^2$}

We suppose that we have two samples at hand, the first being used to construct an estimator $\tilde \theta$, such as the one from (\ref{cplan}). We freeze $\tilde \theta$ and the first sample in what follows and all probabilistic statements are under the distribution $\pp_\theta$ of the second sample $Y,X$ of size $n \in \mathbb N$, conditional on the value of $\tilde \theta$. We define the following residual sum of squares (RSS) statistic
\begin{equation} \label{RSS0}
\hat r_n = \frac{1}{n}\|Y-\mathcal X \tilde \theta\|^2 - \sigma^2,
\end{equation}
which satisfies $\ee _\theta \hat r_n = \|\theta-\tilde \theta\|_F^2$ in the model (\ref{model}) under Conditions \ref{design} and \ref{gausse} (see the proof of Theorem \ref{RSSthm} below). We assume for now that $\sigma$ is known, see Subsection \ref{novar} below for a discussion of the necessary modifications in the general case. Given $\alpha>0$, let $\xi_{\alpha, \sigma}$ be quantile constants such that 
\begin{equation} \label{quantile}
\Pr\left(\sum_{i=1}^n(\varepsilon_i^2-1)>\xi_{\alpha, \sigma} \sqrt n\right) = \alpha
\end{equation}
(these constants converge to the quantiles of a fixed normal distribution as $n \to \infty$), let $z_\alpha=\log(3/\alpha)$ and, for $z \ge 0$ a fixed constant to be chosen, define the confidence set
\begin{equation} \label{RSSconf}
C_n = \left\{v \in \hh_d(\cc): \|v-\tilde \theta\|_F^2 \le  2 \left(\hat r_n + z \frac{d}{n} + \frac{\bar z+\xi_{\alpha/3, \sigma}}{\sqrt n}\right)  \right\},
\end{equation}
where $$\bar z^2= \bar z^2(\alpha,d,n, \sigma, v) = z_{\alpha/3}\sigma^2 \max(3\|v-\tilde \theta\|^2_F, 4zd/n).$$ Note that in the `quantum shape constraint' case $\theta \in \Theta_+$ we can always upper bound $\|v-\tilde\theta\|_F \le 2$ in the definition of $\bar z$, which gives a confidence set that is easier to compute and of only marginally larger overall diameter. In some situations, however, the quantity $\bar z/\sqrt n$ is of smaller order than $1/\sqrt n$, and the more complicated expression above is generally preferable. 

%\smallskip

It is not difficult to see (using that $x^2 \lesssim y+x/\sqrt n$ implies $x^2 \lesssim y +1/n$) that the mean square Frobenius norm diameter of $C_n$ is of order
\begin{equation} \label{diameter}
\ee _\theta|C_n|^2_F\lesssim \|\tilde \theta -\theta\|_F^2 +  \frac{zd + z_{\alpha/3}}{n} + \frac{\xi_{\alpha/3, \sigma}}{\sqrt n}.
\end{equation}
Whenever $d \ge \sqrt n$ -- so as long as at most $n \le d^2$ measurements have been taken -- the deviation terms are of smaller order than $kd/n$ for any $k \ge 1$, and hence $C_n$ has minimax optimal expected squared diameter whenever the estimator $\tilde \theta$ is minimax optimal as in (\ref{cplan}). 

%\smallskip

The following result shows that $C_n$ is a valid confidence set for arbitrary Hermitian $d \times d$ matrices (without any rank constraint). Note that the result is non-asymptotic -- it holds for every $n \in \mathbb N$.

\begin{theorem} \label{RSSthm}
Let $\theta \in \hh_d(\cc)$ be arbitrary and let $\pp_\theta$ be the distribution of $Y,X$ from model (\ref{model}) under Condition \ref{gausse}.

%\smallskip

a) Assume the design satisfies Condition \ref{design}a) and let $C_n$ be given by (\ref{RSSconf}) with $z=0$. We then have for every $n \in \mathbb N$ that $$\pp_\theta(\theta \in C_n) \ge 1-\frac{2\alpha}{3}  - 2 e^{-c n}$$ where $c$ is a numerical constant. In the case of standard Gaussian design, $c=1/24$ is admissible.

%\smallskip

b) Assume the design satisfies Condition \ref{design}b) with constant $K>0$, let $C_n$ be given by (\ref{RSSconf}) with $z>0$ and assume also that $\theta \in \Theta_+$ and $\tilde \theta \in \Theta_+$ (that is, both satisfy the `quantum shape constraint'). Then for every $n \in \mathbb N$,
 $$\pp_\theta(\theta \in C_n) \ge 1-\frac{2\alpha}{3} -  2 e^{-C(K) z}$$ 
where $C(K) =1/[(16+8/3)K^2]$.
\end{theorem}

%\smallskip

In Part a), if we want to control the coverage probability at level $1-\alpha$, $n$ needs to be large enough so that the third deviation term is controlled at level $\alpha/3$. In the Gaussian design case with $\alpha=0.05$, $n \ge 100$ is sufficient, for smaller sample sizes one can use the confidence region from the next subsection. The bound in b) is entirely non-asymptotic for suitable choices of $z$.  Also note that the quantile constants $z, z_\alpha, \xi_\alpha$ all scale at least as $O(\log(1/\alpha))$ in the desired coverage level $\alpha \to 0$.

As mentioned above, the confidence set from Theorem \ref{RSSthm} is optimal whenever the desired performance of $\|\theta-\tilde \theta\|_F^2$ is no better than of order $1/\sqrt n$, corresponding to the important regime $n \le d^2$ for sequential sampling algorithms. Refinements for measurement scales $n\ge d^2$ are also of interest - we present two optimal approaches in the next two subsections for the designs from Condition \ref{design}. 

\subsubsection{Isotropic design and a confidence set based on $U$-statistics} \label{ustats}

Consider isotropic i.i.d~design from Condition \ref{design}a), and an estimator $\tilde \theta$ based on an initial sample of size $n$ (all statements that follow are conditional on that sample) . Collect another $n$ samples to perform the uncertainty quantification step. Define the $U$-statistic
\begin{equation} \label{U0}
\hat R_n = \frac{2}{n(n-1)} \sum_{i<j} \sum_{m,k} (Y_i X^i_{m,k}-\tilde \theta_{m,k}) (Y_j X^j_{m,k}-\tilde \theta_{m,k})
\end{equation}
whose $\ee _\theta$-expectation, conditional on $\tilde \theta$, equals $\|\theta-\tilde\theta\|_F^2$ in view of
$$\ee Y_i X_{m,k}^i = \ee  \sum_{m',k'} X^i_{m',k'}X^i_{m,k} \theta_{m',k'} = \theta_{m,k}.$$ 
Define
\begin{equation} \label{Uconf}
C_n = \left\{v \in \hh_d(\rr): \|v - \tilde \theta\|^2_F \le \hat R_n + z_{\alpha,n}  \right\}
\end{equation}
where $$z_{\alpha, n} = \frac{C_1  \|\theta-\tilde \theta\|_F}{\sqrt n} + \frac{C_2 d}{n} $$ and $C_1 \ge \zeta_1 \|\theta\|_F,~C_2 \ge \zeta_2 \|\theta\|_F^2$ with $\zeta_i$ constants depending on $\alpha$ and the upper bound $v$ for $\sigma$ from Condition \ref{gausse}. Note that if $\theta \in \Theta_+$ then $\|\theta\|_F \le 1$ can be used as an upper bound in $C_i, i=1,2$. In practice the constants $\zeta_i$ can be calibrated by Monte Carlo simulations (see the implementation section below), or chosen based on concentration inequalities for $U$-statistics 
(see ref. \cite{GN15}, Theorem 4.4.8). This confidence set has expected diameter
$$\ee _\theta |C_n|^2_F \lesssim \|\tilde \theta- \theta\|_F^2 +\frac{C_1+C_2d}{n},$$ and hence is compatible with any minimax recovery rate $\|\tilde \theta - \theta\|_F^2 \lesssim kd/n$ from (\ref{cplan}), where $k \ge 1$ is now arbitrary. For suitable choices of $\zeta_i$ we now show that $C_n$ also has non-asymptotic coverage.

\begin{theorem}\label{ustatkill}
Assume Conditions \ref{design}a) and \ref{gausse}, and let $C_n$ be as in (\ref{Uconf}). For every $\alpha>0$ we can choose $\zeta_i(\alpha)=O(\sqrt{1/\alpha}), i=1,2,$ large enough so that for every $n \in \mathbb N$ we have $$\pp_\theta (\theta \in C_n) \ge 1-\alpha.$$ 
\end{theorem}

\subsubsection{Pauli design: Re-averaging basis elements when $n \ge d^2$} \label{reav}

For the design from Condition \ref{design}b) where we sample uniformly at random from a (scaled) basis $\{dE_1, \dots, dE_{d^2}\}$ of $\mm_d(\cc)$, the $U$-statistic approach from Theorem \ref{ustatkill} appears not to be viable, and thus for $d \le \sqrt n$ the existence of an optimal confidence region still needs to be ensured. When $d \le \sqrt n$ we are taking $n \ge d^2$ measurements, and there is no need to sample at \textit{random} from the basis as we can measure each individual coefficient, possibly even multiple times. Repeatedly sampling a basis coefficient $tr(E_k \theta)$ leads to a reduction of the variance of the measurement by averaging. More precisely, when taking $n = md^2$ measurements for some (for simplicity integer) $m\ge 1$, and if $(Y_{k,l}: l =1, \dots, m)$ are the measurements $Y_i$ corresponding to the basis element $E_k, k \in \{1, \dots, d^2\}$, we can form averaged measurements $$Z_k = \frac{1}{\sqrt m} \sum_{l=1}^m Y_{k,l} = \sqrt m d\langle E_{k}, \theta \rangle_F + \epsilon_k, ~~\epsilon_k = \frac{1}{\sqrt m} \sum_{l=1}^{m}\varepsilon_l \sim N(0,\sigma^2).$$ We can then define the new measurement vector $\tilde Z = (\tilde Z_1, \dots, \tilde Z_{d^2})^T$ (using also $m=n/d^2$) 
$$\tilde Z_k = Z_k - \sqrt n\langle \tilde \theta, E_k \rangle =  \sqrt n \langle E_k, \theta - \tilde \theta \rangle_F + \epsilon_k, ~~k =1, \dots, d^2$$ and the statistic
\begin{equation}\label{RSS2}
\hat R_n = \frac{1}{n}\|\tilde Z\|_{\rr^{d^2}}^2 - \frac{\sigma^2d^2}{n}
\end{equation}
which estimates $\|\theta - \tilde\theta\|_F^2$ with precision
\begin{align*}
\hat R_n - \|\theta-\tilde \theta\|_F^2 & = \frac{2}{\sqrt n} \sum_{k=1}^{d^2}\epsilon_k \langle E_k, \theta - \tilde \theta \rangle_F + \frac{1}{n}\sum_{k=1}^{d^2}(\epsilon^2_k-\ee \epsilon^2) \\
&= O_P\left(\frac{\sigma\|\theta-\tilde \theta\|_F}{\sqrt n} + \frac{\sigma^2 d}{n}\right) \notag.
\end{align*}
Hence, for $z_{\alpha}$ the quantiles of a $N(0,1)$ distribution and $\xi_{\alpha, \sigma}$ as in (\ref{quantile}) with $d^2$ replacing $n$ there, we can define a confidence set
\begin{equation} \label{liconf}
\bar C_n = \left\{v \in \hh_d(\cc): \|v - \tilde \theta\|_F^2 \le \hat R_n + \frac{z_{\alpha/2} \sigma \|\theta -\tilde \theta\|_F}{\sqrt n} + \frac{\xi_{\alpha/2, \sigma} d}{n} \right\}
\end{equation}
which has non-asymptotic coverage
$$\pp_\theta (\theta \in \bar C_n) \ge 1- \alpha$$
for every $n \in \mathbb N$, by similar (in fact, since Lemma \ref{bernstein} is not needed, simpler) arguments as in the proof of Theorem \ref{RSSthm} below. The expected diameter of $\bar C_n$ is by construction 
\begin{equation}  \label{diameter2}
\ee _\theta |\bar C_n|^2_F \lesssim \|\theta - \tilde \theta\|_F^2 + \frac{\sigma^2d}{n},
\end{equation}
now compatible with \textit{any} rate of recovery $kd/n, 1\le k \le d$. The case of unknown variance is discussed in the next subsection.

\subsubsection{Unknown variance} \label{novar}

The $U$-statistic based confidence set from (\ref{Uconf}) does not require knowledge of $\sigma$ but works only for the design from Condition \ref{design}a). For Pauli design from Condition \ref{design}b) we can use the confidence sets $C_n$ in Theorem \ref{RSSthm} or $\bar C_n$ in (\ref{liconf}), but these do require exact knowledge of the noise variance $\sigma^2$. As described before (\ref{bernbds}) above, in the Pauli case $\sigma^2$ can be apriori bounded by $d/T$, where $T$ is the number of preparations used to measure each individual Pauli observable. If $T \ge n$ then the statistics $\hat r_n$ and $\hat R_n$ from (\ref{RSS0}) and (\ref{RSS2}) above can be used without subtracting $\sigma^2$ and $\sigma^2d^2/n$, respectively, in their definitions. The coverage proofs then go through with minor modifications simply by noting that these centerings are of sufficiently small order of magnitude $\sigma^2 \le d/T \le d/n$ and $\sigma^2 d^2/n \le d^3/Tn \le d/n$ compared to the minimax rate of estimation, and by using the upper bound $\sigma^2 \le v= d/T $ in all relevant constants featuring in the definition of $C_n, \bar C_n$. 

Typically preparing $T \ge n$ measurements of a fixed Pauli observable is not a major problem in experimental situations. If for some reason this cannot be done, one can make sure that each $tr(E_i\theta)$ is at least measured twice (so $T \ge 2$), say in batches $Y_1, \dots, Y_{n/2}$ and $Y_{n/2+1}, \dots, Y_n$, and then use the modified statistic $$\tilde r_n = \frac{2}{n} \sum_{i=1}^{n/2} (Y_i - \langle X^i, \tilde \theta \rangle_F) (Y_{i+n/2} - \langle X^i, \tilde \theta \rangle_F)$$ in the construction of the confidence set. Arguments similar to above, using concentration inequalities for Gaussian chaos of order two (Theorem 3.1.9 in \cite{GN15}) then allow for the construction of a confidence region that does neither require knowledge of $\sigma^2 \le v$ nor $T \ge n$. Details are omitted.

%\begin{remark} [Bernoulli noise] \label{bernmodif} \normalfont It is not difficult to see that Theorem \ref{RSSthm}b) holds as well for the Bernoulli measurement model from Subsection \ref{bernoulli} with $T \ge d^2$, with slightly different constants in the construction of $C_n$ and the coverage probabilities, see Remark \ref{bernprf} after the proof of Theorem \ref{RSSthm}b) below for details. A similar remark applies to the coverage proof for the confidence region described in Subsection \ref{reav}. The modified quantile constants $z, z_\alpha, \xi_\alpha$ still scale as $O(\sqrt{1/\alpha})$ in the desired coverage level $\alpha \to 0$, and hence the adaptive sampling Theorem \ref{adsamp} holds for such noise too, if the number $T$ of preparations of the quantum state exceeds $d^2$. \end{remark}

\subsection{A confidence set in trace norm under quantum shape constraints}

The confidence sets from the previous subsections are all valid in the sense that they contain information about the recovery of $\theta$ by $\tilde \theta$ in Frobenius norm $\|\cdot\|_F$. It is of interest to obtain results in stronger norms, such as for instance the nuclear norm $\|\cdot\|_{S_1}$, which is particularly meaningful for quantum tomography problems since it then corresponds to the total variation distance on the set of `probability density matrices'. The absence of the `Hilbert space geometry' induced by the relationship of the Frobenius norm to the inner product $\langle \cdot, \cdot \rangle_F$ makes this problem significantly harder, both technically and from an information-theoretic point of view. In particular the quantum shape constraint $\theta \in \Theta_+$ is crucial to obtain any results whatsoever. For the theoretical results presented here it will be more convenient to perform an asymptotic analysis where $\min(n,d) \to \infty$ (with $o,O$-notation to be understood accordingly).

%\medskip

Instead of Condition \ref{design} we now consider more generally any design $(X^i, i=1, \dots, n)$ in model (\ref{model}) that satisfies the matrix RIP (\ref{RIP}) with
\begin{equation} \label{MRIP}
\tau_n (k) = c \sqrt{kd \frac{\overline{\log} (d)}{n}}.
\end{equation}
We shall still use the convention discussed before Condition \ref{design} that $\theta$ and the matrices $X^i$ are such that $tr(X^i\theta)$ is always real-valued. 

In contrast to the results from the previous section we shall now assume a minimal low rank constraint on the parameter space:
\begin{condition} \label{consistent} $\theta \in R^+(k) := R(k) \cap \Theta_+$ for some $k$ satisfying
$$k\sqrt{\frac{d \overline{\log} d}{n}} = o(1),$$
\end{condition}
This in particular implies that the RIP holds with $\tau_n(k) = o(1)$.  Given this minimal rank constraint $\theta  \in R^+(k)$, we now show that it is possible to construct a confidence set $C_n$ that adapts to any low rank $1 \le k_0<k$. Here we may choose $k=d$ but note that this forces $n \gg d^2$ (for Condition \ref{consistent} to hold with $k=d$). 

%\smallskip

We assume that there exists an estimator $\tilde \theta_{\rm Pilot}$ that satisfies, uniformly in $R(k_0)$ for any $k_0 \le k$ and for $n$ large enough,
\begin{equation} \label{risk}
\|\tilde \theta_{\rm Pilot} - \theta\|_F^2 \leq D \sigma^2 \frac{k_0 d }{n} := \frac{r^2_n(k_0)}{4}
\end{equation}
where $D=D(\delta)$ depends on $\delta$, and where so-defined $r_n$ will be used frequently below. Such estimators exist as has already been discussed before (\ref{cplan}). We shall in fact require a little more, namely the following oracle inequality: for any $k$ and any matrix $S$ of rank $k\le d$, with high probability and for $n$ large enough,
\begin{equation} \label{oracle}
\|\tilde \theta_{\rm Pilot} - \theta\|_F \lesssim \|\theta - S\|_F + r_n(k),
\end{equation}
which implies (\ref{risk}). Such inequalities exist assuming the RIP and Condition \ref{consistent}, see, e.g., Theorem 2.8 in ref.  \cite{CP11}. Starting from $\tilde \theta_{\rm Pilot}$ one can construct (see Theorem \ref{estcond} below) an estimator that recovers $\theta \in R(k)$ in nuclear norm at rate $k \sqrt {d/n}$, which is again optimal from a minimax point of view, even under the quantum constraint (as discussed, e.g., in ref. \cite{K11}). We now construct an adaptive confidence set for $\theta$ centred at a suitable projection of $\tilde \theta_{\rm Pilot}$ onto $\Theta_+$.

%\medskip

In the proof of Theorem \ref{main} below we will construct estimated eigenvalues $(\hat \lambda_j, j=1, \dots, d)$ of $\theta$ (see after Lemma \ref{pcaell1}). Given those eigenvalues and $\tilde \theta_{\rm Pilot}$, we choose $\hat k$ to equal the smallest integer $\le d$ such that there exists a rank $\hat k$ matrix $\tilde \theta'$ for which $$\|\tilde \theta' - \tilde \theta_{\rm Pilot}\|_F \le r_n(\hat k) \text{ and } 1 - \sum_{J \le \hat k} \hat \lambda_J \leq 2 \hat k \sqrt{d/n}$$ is satisfied. Such $\hat k$ exists with high probability (since the inequalities are satisfied for the true $\theta$ and $\lambda_j$'s, as our proofs imply).  Define next $\hat \vartheta$ to be the $\langle \cdot, \cdot \rangle_F$-projection of $\tilde \theta_{\rm Pilot}$ onto $$R^+(2\hat k) 
:= R(2 \hat k) \cap \Theta_+$$ and note that, since $2 \hat k \ge \hat k$, 
\begin{equation} \label{projk}
\|\tilde \theta_{\rm Pilot} - \hat \vartheta\|_F =\|\tilde \theta_{\rm Pilot} - R^+(2\hat k)\|_F \le \|\tilde \theta_{\rm Pilot} - \tilde \theta'\|_F \le r_n(\hat k).
\end{equation}
Finally define, for $C$ a constant chosen below,
\begin{equation}
C_n = \left\{v \in \Theta_+ : \|v - \hat \vartheta\|_{S_1} \le C \sqrt{\hat k} r_n(\hat k)  \right\}.
\end{equation}
\begin{theorem}\label{main}
Assume Condition \ref{consistent} for some $1 \le k \le d$, and let $\delta>0$ be given. Assume that with probability greater than $1-2\delta/3$, a) the RIP (\ref{RIP}) holds with $\tau_n(k)$ as in (\ref{MRIP}) and b) there exists an estimator $\tilde \theta_{\rm Pilot}$ for which (\ref{oracle}) holds. Then we can choose $C=C(\delta)$ large enough so that, for $C_n$ as in the last display, $$\liminf_{\min (n,d) \to \infty} \inf_{\theta \in R^+(k)}\pp_\theta (\theta \in C_n) \ge 1-\delta.$$ Moreover, uniformly in $R^+(k_0), 1 \le k_0 \le k,$ and with $\pp_\theta$-probability greater than $1-\delta$, $$|C_n|_{S_1} \lesssim \sqrt {k_0} r_n(k_0).$$  
\end{theorem}

%\smallskip

% Note that for any $\delta$ the requirements a) and b) can be met by choosing the constants in (\ref{MRIP}) and (\ref{oracle}) large enough depending only on $\delta$ and the maximal rank $k$. Hence $C_n$ is a genuine confidence set (in that it can be computed from the data alone in principle). 

%\smallskip

Theorem \ref{main} shows how the quantum shape constraint allows for the construction of an optimal nuclear norm confidence set that adapts to the unknown low rank structure. A careful study of certain hypothesis testing problems (combined with lower bound techniques for confidence sets as in \cite{HN11, NvdG13}) shows that the assumption $\theta \in R^+(k)$ in the above theorem is actually necessary, and cannot be relaxed to $\theta \in R(k)$. See \cite{CN15}, Theorem 4.

\subsection{Conclusions}

We have constructed adaptive confidence regions for matrix parameters $\theta$ in the trace regression model (\ref{model}). These confidence regions contract at the minimax optimal rates for low rank parameters, either in Frobenius or nuclear norm, and are `honest' (in the sense of \cite{L89}, see also \cite{RV06, HN11}). The conditions employed are naturally compatible with quantum tomography applications - where $\theta$ is the 
density matrix of a quantum state, and where the noise variance has an a priori upper bound that can be controlled experimentally. This in turn can be used to demonstrate the existence of fully adaptive sequential sampling protocols that generate valid certificates for the recovery of unknown low rank quantum states.

%\smallskip

While it can be shown on the one hand (see Theorem 4 in \cite{CN15}) that our results for the nuclear norm (Theorem \ref{main}) fundamentally rely on the `quantum shape constraint' $\theta \in \Theta_+$, our results for the Frobenius norm on the other hand are valid in a general compressed sensing inference setting. This may seem surprising in light of negative results in \cite{NvdG13}, where it is shown that in the related `sparse' high-dimensional linear model, signal strength assumptions (inspired by the nonparametric statistics literature, \cite{GN10, HN11}) are generally necessary for the existence of $\ell_2$-confidence regions for the entire parameter vector. However, the information theoretic structure of the matrix inference problem is different, as is also illustrated by the fact that the signal detection rates in the model (\ref{model}) in Frobenius norm do \textit{not} depend on the low rank structure at all (see Theorem 1 in \cite{CN15}). In this sense, our findings in the matrix regression model form a remarkable exception to the rule that uncertainty quantification methodology does not generally exist for high-dimensional adaptive algorithms, unless one restricts the inferential interest to a simple semi-parametric low-dimensional functional (\cite{vdGBRD14, ZZ14, JM14, CK15}).

\section{Simulation experiments} \label{simul}

In order to illustrate the methods from this paper, we present some numerical simulations. The setting of the experiments is as follows: A random matrix $\eta\in \mm_d(\cc)$
of norm $\|\eta\|_F=R^{1/2}$ is generated according to two distinct procedures that we will specify later, and the observations are now
$$\bar {Y}_i = \mathrm{tr}(X^i \eta) + \varepsilon_i.$$
where the $\varepsilon_i$ are i.i.d.~Gaussian of mean $0$ and variance $1$. The observations are reparametrised so that $\eta$ represents the `estimation error' $\theta - \hat \theta$, and we investigate how well the statistics $$\hat r_n = \frac{1}{n} \|\bar Y\|- 1 \text{ and } \hat R_n = \frac{2}{n(n-1)} \sum_{i<j} \sum_{m,k} \bar Y_i X^i_{m,k} \bar Y_j X^j_{m,k}$$ estimate the `accuracy of estimation' $\|\eta\|_F^2= \|\theta-\hat \theta\|_F^2$, conditional on the value of $\hat \theta$. We will choose $\eta$ in order to illustrate two extreme cases: a first one where the nuclear norm $\|\eta\|_{S_1}$ is `small', corresponding to a situation where the quantum constraint is fulfilled; and a second one where the nuclear norm is large, corresponding to a situation where the quantum constraint is \textit{not} fulfilled. More precisely we generate the parameter $\eta$ in two ways:
\begin{itemize}
\item `Random Dirac' case: set a single entry (with position chosen at random on the diagonal) of $\eta$ to $R^{1/2}$, and all the other coordinates equal to $0$. 
\item `Random Pauli' case: Set $\eta$ equal to a Pauli basis element chosen uniformly at random and then multiplied by $R^{1/2}$.
\end{itemize}
The designs that we consider are the Gaussian design, and the Pauli design, described in Condition 1. We perform experiments with $d=32$, $R \in \{0.1, 1\}$ and $n \in \{100, 200,500,1000,2000,5000\}.$ Note that $d^2=1024$, so that the first four choices of $n$ correspond to the important regime $n<d^2$.  Our results are plotted as a function of the number $n$ of samples  in Figures~\ref{fig:GD}, \ref{fig:GP}, \ref{fig:PD}, \ref{fig:PP}. The solid red an blue curves are the median errors of the normalised estimation errors $$\frac{\sqrt{\hat R_n - R}}{R^{1/2}}, \quad \mathrm{and} \quad \frac{\sqrt{\hat r_n - R}}{R^{1/2}},$$ after $1000$ iterations, and the dotted lines are respectively, the (two-sided) $90\%$ quantiles. We also report  (see Tables~\ref{tab:GD}, \ref{tab:GP}, \ref{tab:PD}, \ref{tab:PP}) how well the confidence sets based on these estimates of the norm perform in terms of coverage probabilities, and of diameters.  The diameters are computed as
$$\left({\hat R_n + \frac{C_{\rm UStat} d}{n} + \frac{C_{\rm UStat}' \hat R_n^{1/2}}{\sqrt{n}} }\right)^{1/2},$$
for the U-Statistic approach and
$$\left({\hat r_n + \frac{C_{\rm RSS}}{\sqrt{n}} + \frac{C_{\rm RSS}' \hat r_n^{1/2}}{\sqrt{n}}}\right)^{1/2},$$
for the RSS approach, where we have chosen $C_{\rm UStat} = 2.5$, $C_{\rm RSS} = 1$ and $C_{\rm UStat}' = C_{\rm RSS} = 6$ for all experiments --calibrated to a $95\%$ coverage level.  From these numerical results, several observations can be made: 

1) In Gaussian random designs, the results are insensitive to the nature of $\eta$ (see Figures~\ref{fig:GD} and~\ref{fig:GP} and Tables~\ref{tab:GD} and~\ref{tab:GP}). This is not surprising since the Gaussian design is `isotropic'.

2) For Pauli designs with the quantum constraint (see Figure~\ref{fig:PD} and Table~\ref{tab:PD}) the RSS method works quite well even for small sample sizes. But the U-Stat method is not very reliable -- indeed we see no empirical evidence that Theorem \ref{ustatkill} should also hold true for Pauli design.

3) For Pauli design and when the quantum shape constraint is \textit{not} satisfied our methods cease to provide reliable results (see Figure~\ref{fig:PP} and in particular Table~\ref{tab:PP}). Indeed, when the matrix $\eta$ is chosen itself as a random Pauli (which is the hardest signal to detect under Pauli design) both the RSS and the U-Stat approach perform poorly. The confidence set are not honest anymore, which is in line with the theoretical limitations we observe in Theorem \ref{RSSthm}. Figure~\ref{fig:PP} illustrates that the methods do not detect the signal, since the norm of $\eta$ is largely under-evaluated for small sample sizes. These limitations are less pronounced when $n \ge d^2$. In this case one could use alternatively the re-averaging approach from Subsection \ref{reav} (not investigated in the simulations) to obtain honest results without the quantum shape constraint.

\begin{center}
\begin{figure}
\begin{minipage}{0.55\textwidth}
 \includegraphics[width=1 \textwidth]{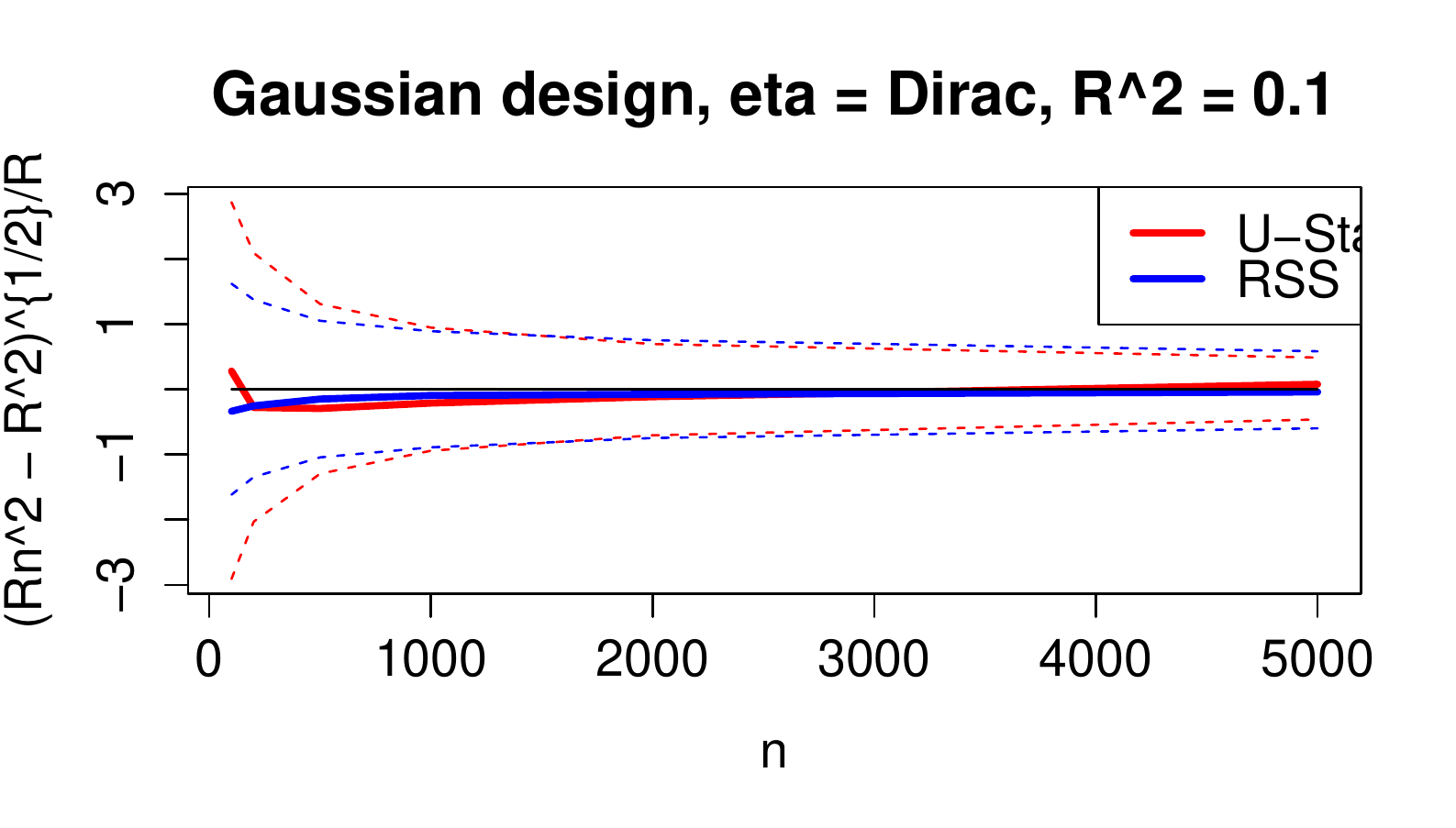}
\end{minipage}
\begin{minipage}{0.55\textwidth}
 \includegraphics[width=1 \textwidth]{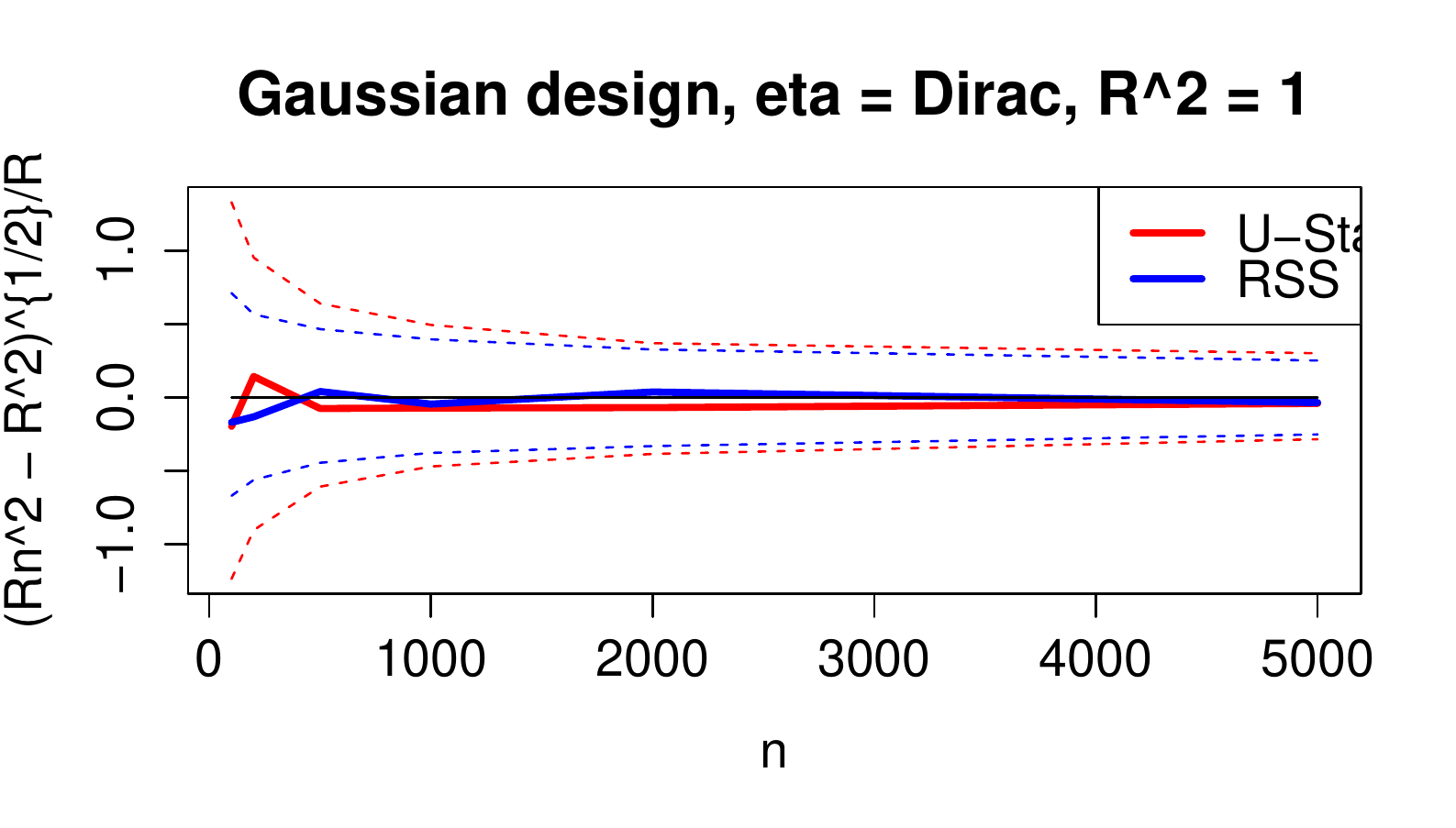}
\end{minipage}
\caption{Gaussian design, and random Dirac (a single entry, chosen at random, is non-zero on the diagonal) $\eta$, with $R = 0.1$ (left picture) and $R = 1$ (right picture).} \label{fig:GD}
\end{figure}
\end{center}

\begin{table}[h]
\begin{center}
\begin{footnotesize}
\begin{tabular}{|c||c|c|c|c|c|c||c|c|c|c|c|c|}
 \hline
 &\multicolumn{6}{|c||}{$R = 0.1$}&\multicolumn{6}{|c|}{$R = 1$}\\
\hline
\hline
 $n$ & $100$ & $200$ & $500$ & $1000$ & $2000$ & $5000$ & $100$ & $200$ & $500$ & $1000$ & $2000$ & $5000$\\
 \hline
\hline
 Coverage U-Stat & 0.97 & 0.98 &0.99 &1.00 & 1.00 & 1.00 & 0.93 &0.96 &0.97 & 0.98 &0.98 &0.98\\
 \hline
 Diameter U-Stat & 1.10 &0.64 &0.34 &0.24 & 0.18 &0.14 & 2.43 & 1.84 & 1.44 & 1.27 &1.17 & 1.10\\
 \hline
 \hline
 Coverage RSS &0.97 &0.97 &0.98 &0.98 &0.98 &0.98 & 0.99 & 0.99 & 0.99 & 0.99 & 0.99 & 0.99\\
 \hline
 Diameter RSS & 0.38 &0.31 &0.23 &0.19 &0.16 & 0.14& 1.69 & 1.49 &1.32 &1.22 &1.16 &1.10\\
 \hline
\end{tabular}
\end{footnotesize}
\end{center}
\caption{Gaussian design, and random Dirac (a single entry, chosen at random, is non-zero on the diagonal) $\eta$, with $R = 0.1$ (left table) and $R= 1$ (right table).} \label{tab:GD}
\end{table}

\begin{center}
\begin{figure}
\begin{minipage}{0.55\textwidth}
\includegraphics[width=1\textwidth]{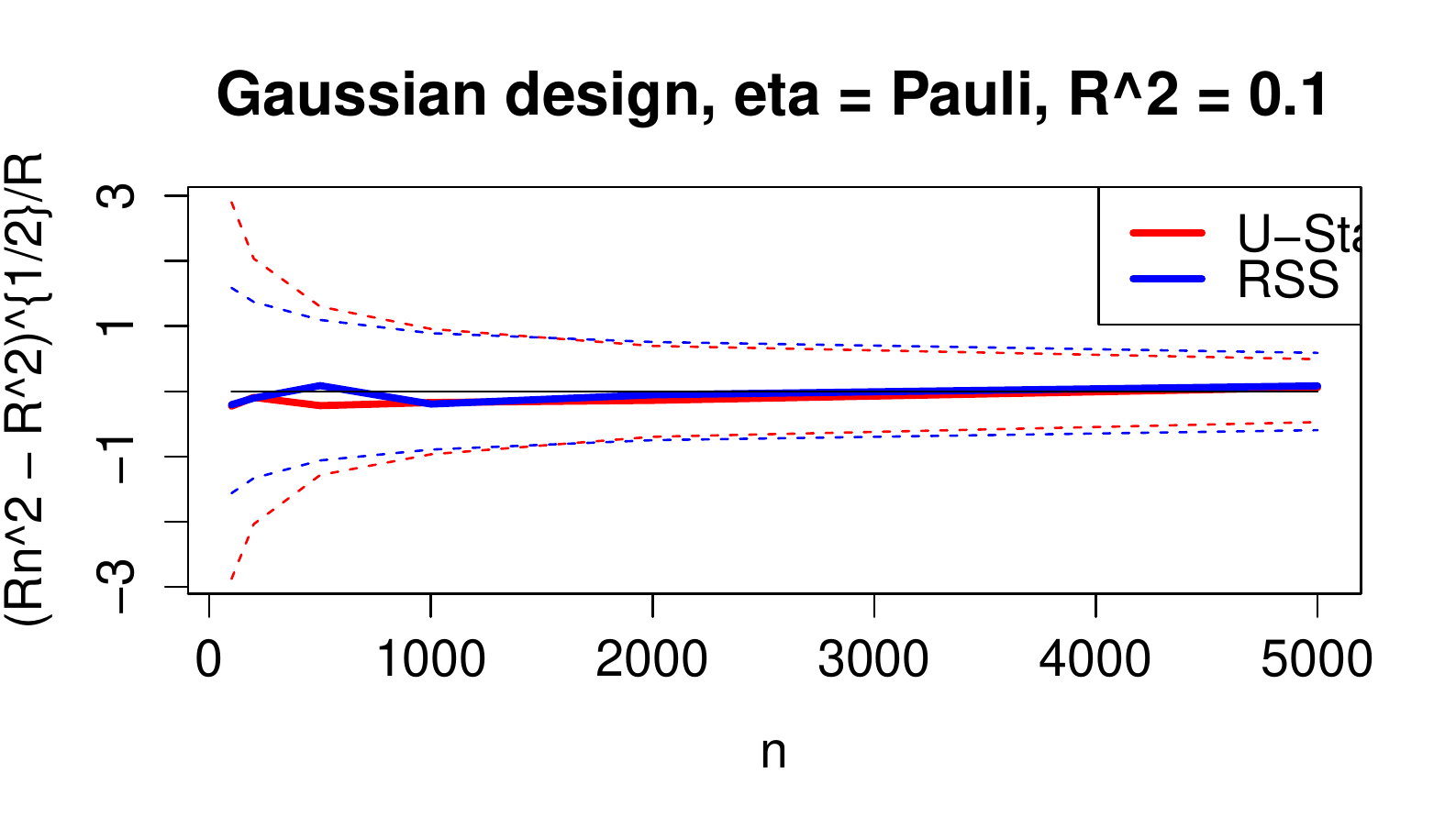}
\end{minipage}
\begin{minipage}{0.55\textwidth}
 \includegraphics[width=1\textwidth]{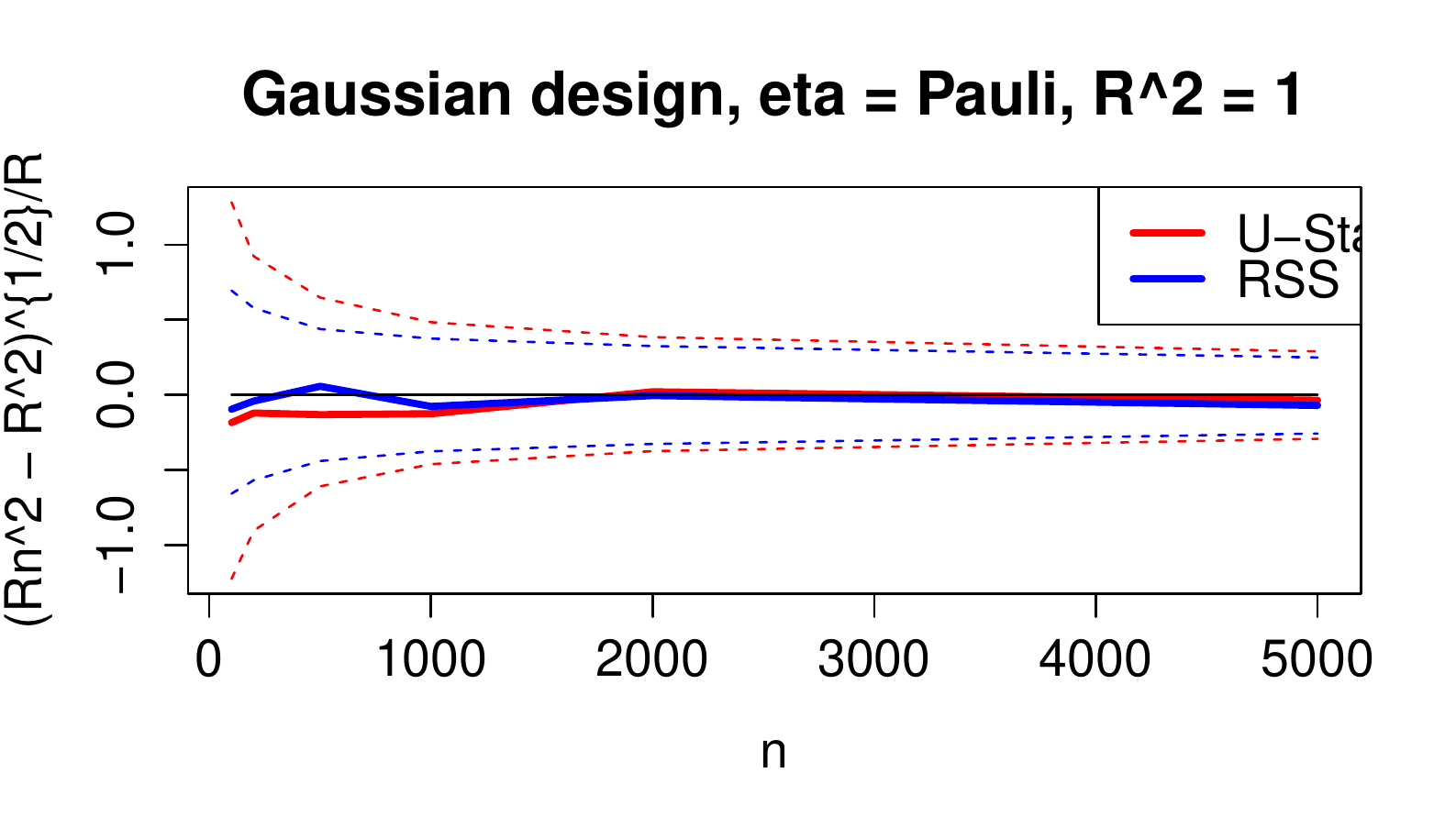}
\end{minipage}
\caption{Gaussian design, and random Pauli $\eta$, with $R = 0.1$ (left picture) and $R = 1$ (right picture).} \label{fig:GP}
\end{figure}
\end{center}

\begin{table}[h]
\begin{center}
\begin{footnotesize}
\begin{tabular}{|c||c|c|c|c|c|c||c|c|c|c|c|c|}
 \hline
 &\multicolumn{6}{|c||}{$R = 0.1$}&\multicolumn{6}{|c|}{$R = 1$}\\
\hline
\hline
 $n$ & $100$ & $200$ & $500$ & $1000$ & $2000$ & $5000$ & $100$ & $200$ & $500$ & $1000$ & $2000$ & $5000$\\
 \hline
\hline
 Coverage U-Stat & 0.98 & 0.98 & 0.99 & 0.99 & 1.0 &1.0 & 0.93 & 0.95 & 0.97 & 0.98 & 0.98 &0.98\\
 \hline
 Diameter U-Stat & 1.10 & 0.62& 0.34 & 0.24 & 0.18 & 0.14 & 2.40  & 1.83 & 1.43 & 1.27 & 1.18 & 1.10\\
 \hline
 \hline
 Coverage RSS &0.98 & 0.98 & 0.97 & 0.97 & 0.97 & 0.97 & 0.99 & 0.99 & 0.99 & 0.99 & 1.00 & 1.00\\
 \hline
 Diameter RSS &0.39 & 0.31 & 0.23 & 0.19 & 0.17 & 0.14 & 1.71 & 1.49 & 1.31 & 1.22 & 1.16 &1.10\\
 \hline
\end{tabular}
\end{footnotesize}
\end{center}
\caption{Gaussian design, and random Pauli $\eta$, with $R = 0.1$ (left table) and $R = 1$ (right table).} \label{tab:GP}
\end{table}

\begin{center}
\begin{figure}
\begin{minipage}{0.55\textwidth}
 \includegraphics[width=1\textwidth]{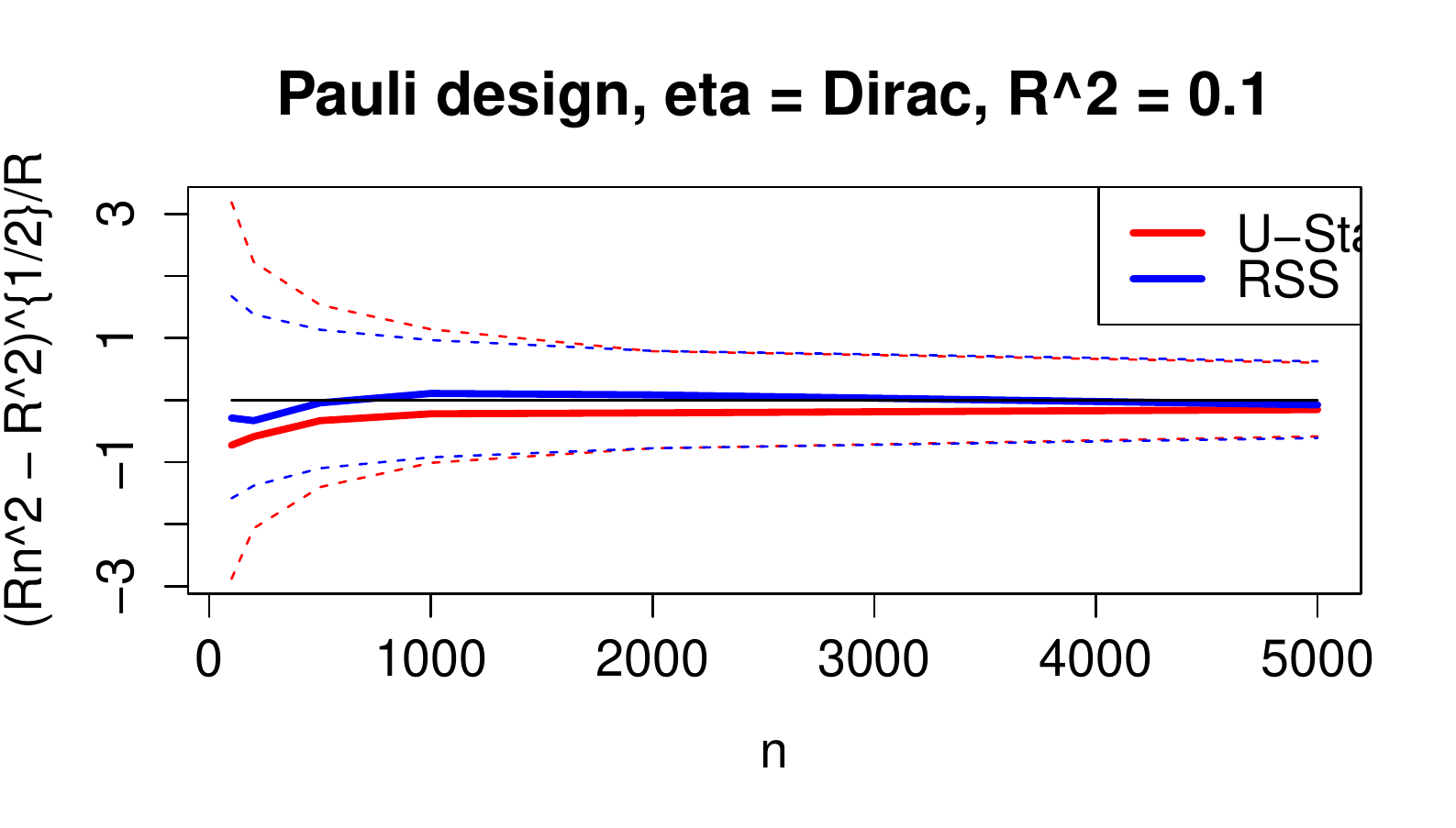}
\end{minipage}
\begin{minipage}{0.55\textwidth}
 \includegraphics[width=1\textwidth]{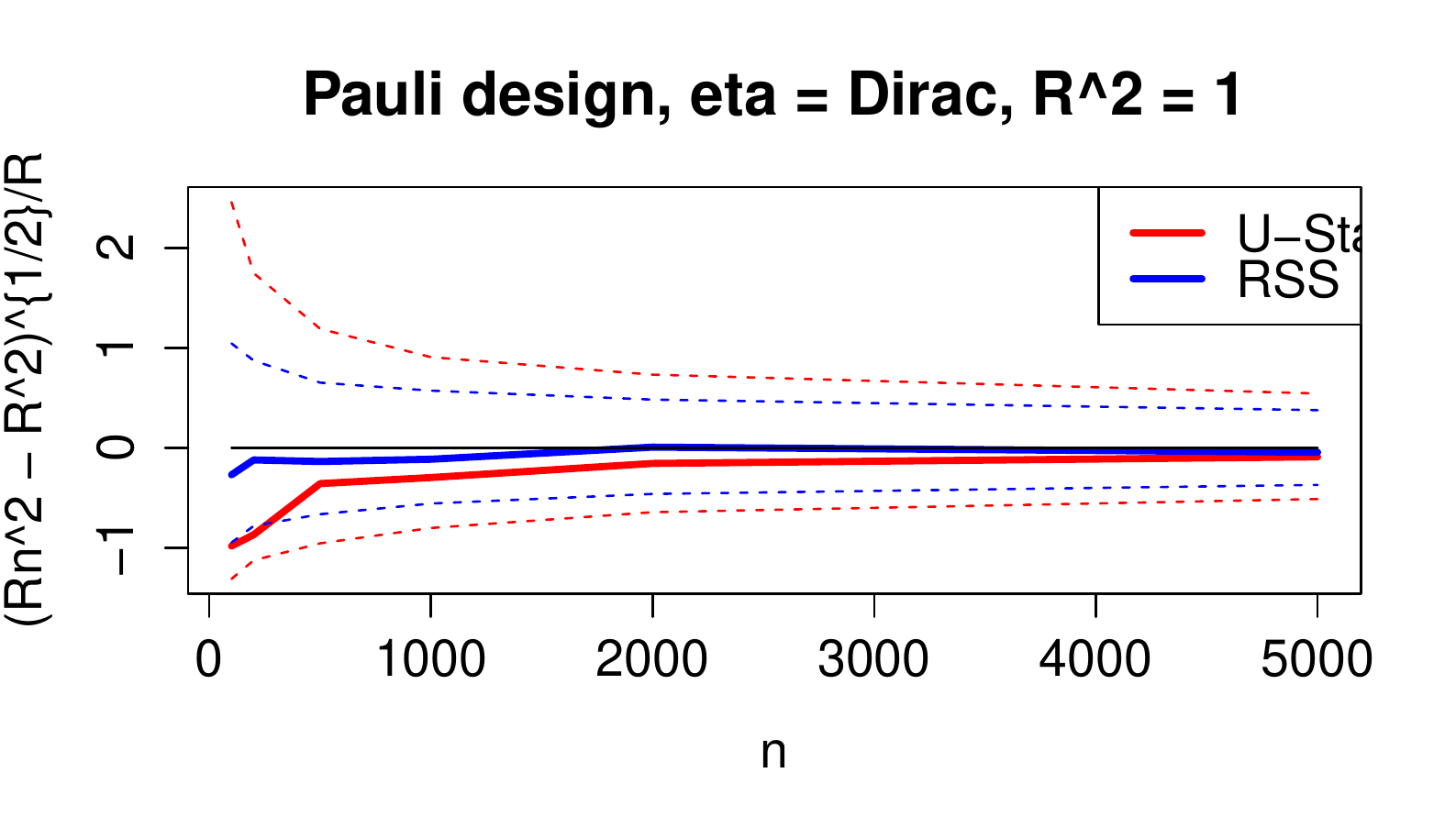}
\end{minipage}
\caption{Pauli design, and random Dirac (a single entry, chosen at random, is non-zero on the diagonal) $\eta$, with $R = 0.1$ (left picture) and $R = 1$ (right picture).} \label{fig:PD}
\end{figure}
\end{center}

\begin{table}[h]
\begin{center}
\begin{footnotesize}
\begin{tabular}{|c||c|c|c|c|c|c||c|c|c|c|c|c|}
 \hline
 &\multicolumn{6}{|c||}{$R = 0.1$}&\multicolumn{6}{|c|}{$R = 1$}\\
\hline
\hline
 $n$ & $100$ & $200$ & $500$ & $1000$ & $2000$ & $5000$ & $100$ & $200$ & $500$ & $1000$ & $2000$ & $5000$\\
 \hline
\hline
 Coverage U-Stat & 0.97 &0.98 & 0.98 &0.99 & 0.98 &0.98& 0.85 & 0.54 & 0.69 & 0.69 & 0.70 & 0.71\\
 \hline
 Diameter U-Stat &  1.10 & 0.63 & 0.34 & 0.24 &0.18 &0.14& 2.28 & 1.87 & 1.43 & 1.26 & 1.18 & 1.10\\
 \hline
 \hline
 Coverage RSS & 0.96 & 0.96 & 0.96 & 0.96 & 0.97 & 0.97& 0.88 & 0.89 & 0.88 & 0.88 & 0.88 & 0.88\\
 \hline
 Diameter RSS & 0.39 & 0.29 & 0.23 & 0.19 & 0.16 &0.14& 1.70 & 1.50 & 1.30 & 1.21 & 1.16 & 1.10 \\
 \hline
\end{tabular}
\end{footnotesize}
\end{center}
\caption{Pauli design, and random Dirac (a single entry, chosen at random, is non-zero on the diagonal) $\eta$, with $R = 0.1$ (left table) and $R = 1$ (right table).} \label{tab:PD}
\end{table}

\begin{center}
\begin{figure}
\begin{minipage}{0.55\textwidth}
 \includegraphics[width=1\textwidth]{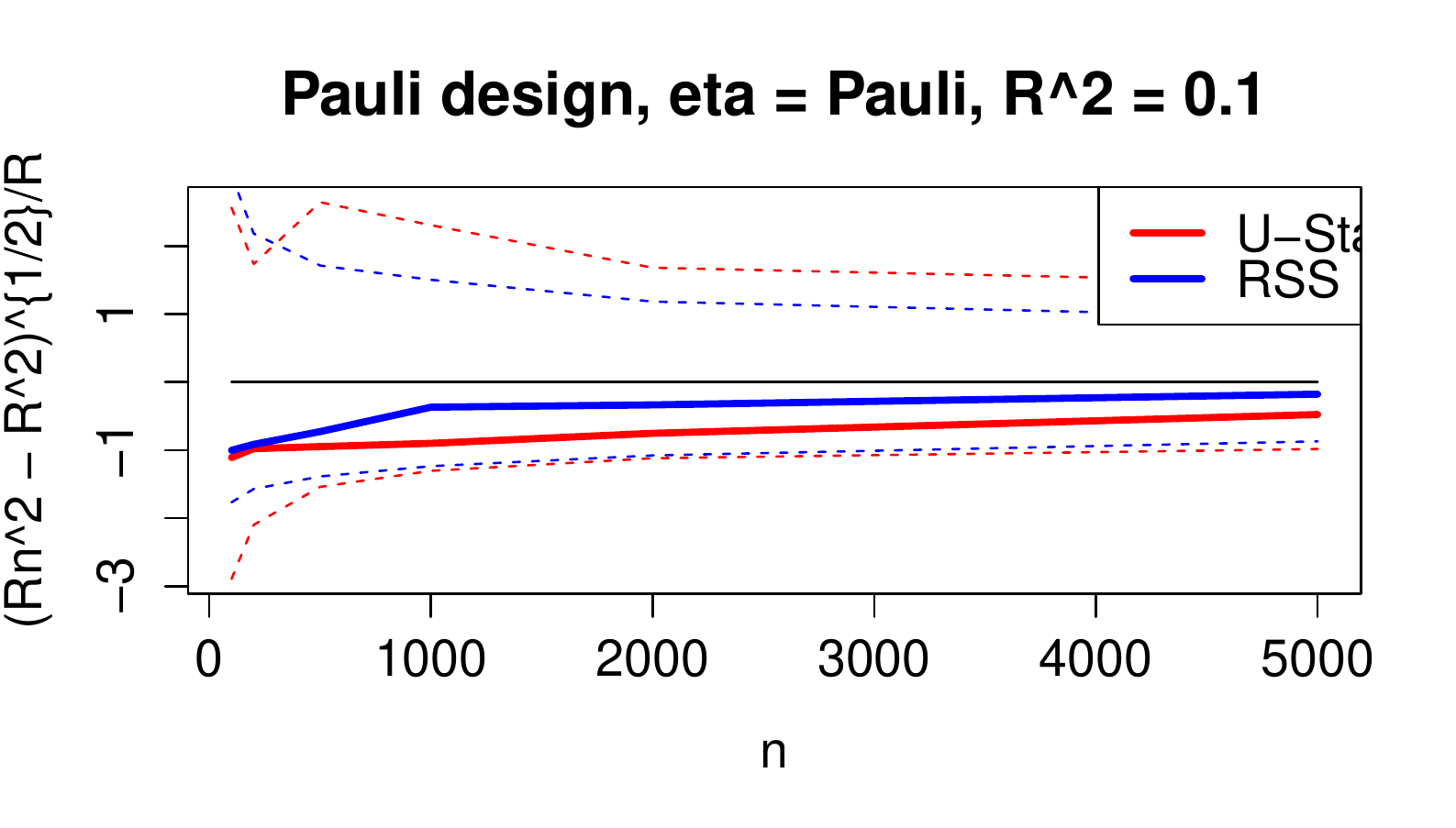}
\end{minipage}
\begin{minipage}{0.55\textwidth}
 \includegraphics[width=1\textwidth]{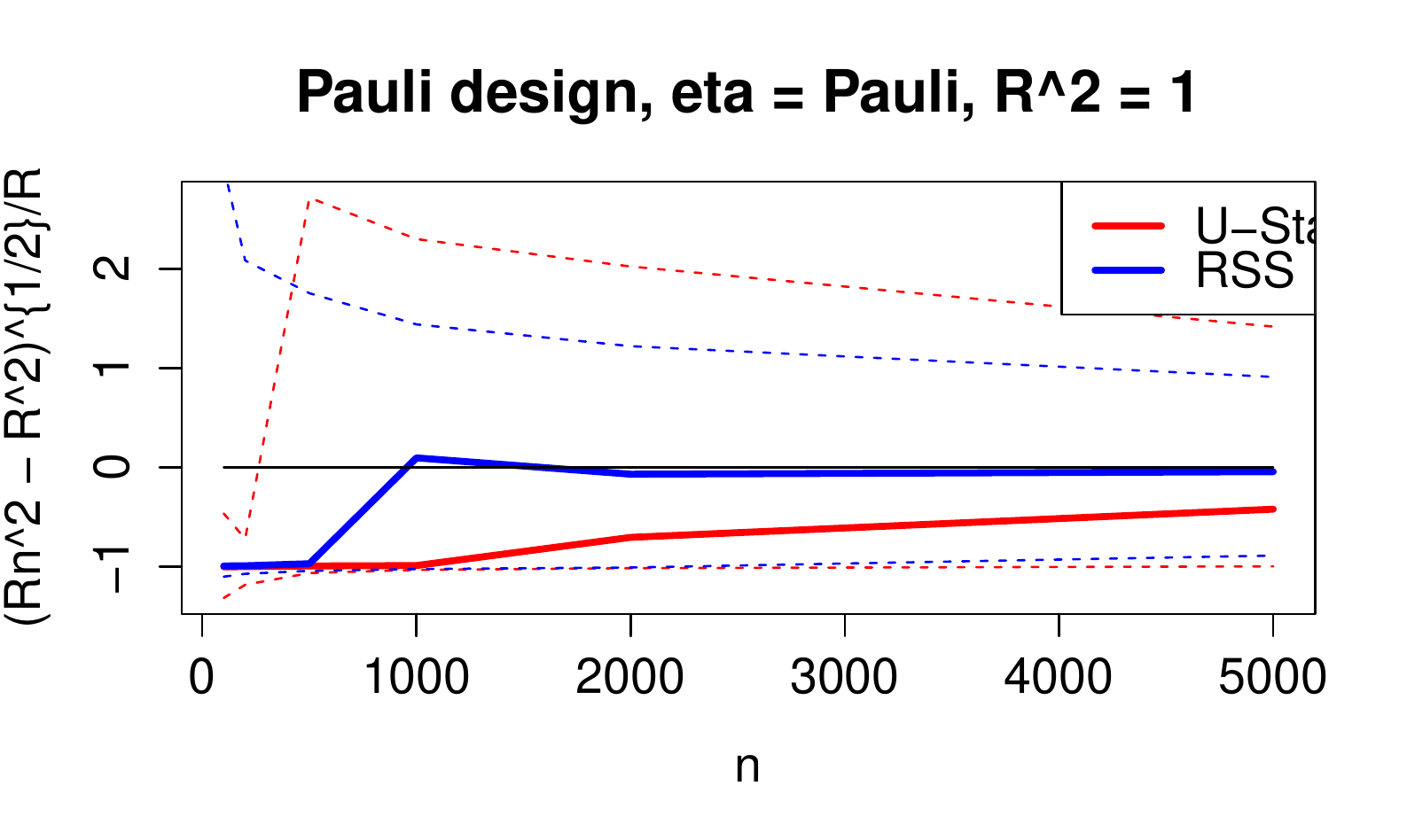}
\end{minipage}
\caption{Pauli design, and random Pauli $\eta$, with $R = 0.1$ (left picture) and $R = 1$ (right picture).} \label{fig:PP}
\end{figure}
\end{center}

\begin{table}[h]
\begin{center}
\begin{footnotesize}
\begin{tabular}{|c||c|c|c|c|c|c||c|c|c|c|c|c|}
 \hline
 &\multicolumn{6}{|c||}{$R = 0.1$}&\multicolumn{6}{|c|}{$R = 1$}\\
\hline
\hline
 $n$ & $100$ & $200$ & $500$ & $1000$ & $2000$ & $5000$ & $100$ & $200$ & $500$ & $1000$ & $2000$ & $5000$\\
 \hline
\hline
 Coverage U-Stat & 0.97 & 0.97 & 0.96 & 0.86 & 0.65 & 0.58 & 0.82 & 0.22 & 0.25 & 0.27 & 0.30 & 0.37\\
 \hline
 Diameter U-Stat & 1.09 & 0.57 & 0.34 & 0.25 & 0.18 & 0.15 & 2.45 & 2.09 & 1.33 & 1.38 & 1.19 & 1.09\\
 \hline
 \hline
 Coverage RSS & 0.93 & 0.86 & 0.77 & 0.77 & 0.77 & 0.77 & 0.12 &0.19 &0.40 & 0.63 &0.56 & 0.53\\
 \hline
 Diameter RSS &  0.38 & 0.29 & 0.22 & 0.19 & 0.16 & 0.14 & 1.71 & 1.56 & 1.31 & 1.26 & 1.14 & 1.08\\
 \hline
\end{tabular}
\end{footnotesize}
\end{center}
\caption{Pauli design, and random Pauli $\eta$, with $R = 0.1$ (left table) and $R =1$ (right table).} \label{tab:PP}
\end{table}

\section{Proofs}

\subsection{Proof of Theorem \ref{adsamp}}

Before we define the algorithm and prove the result, a few preparatory remarks are required: Our sequential procedure will be implemented in $m=1, 2, \dots, T$ potential steps, in each of which $2 \cdot 2^m = 2^{m+1}$ measurements are taken. The arguments below will show that we can restrict the search to at most $$T = O(\log (d/\epsilon))$$ steps. We also note that from the discussion after (\ref{RIP}) -- in particular since $c=c(\delta)$ from (\ref{tau}) is $O(1/\delta^2)$ -- a simple union bound over $m \le T$ implies that the RIP holds with probability $\ge 1-\delta', \text{ some } \delta'>0$, \textit{simultaneously} for every $m \le T$ satisfying $2^m \ge c'kd \overline{\log }d$, and with $\tau_{2^m}(k)<c_0$, where $c'$ is a constant that depends on $\delta', c_0$ only. The maximum over $T=O(\log (d/\epsilon))$ terms is absorbed in a slightly enlarged poly-log term. Hence, simultaneously for all such sample sizes $2^m, m \le T$, a nuclear norm regulariser exists that achieves the optimal rate from (\ref{cplan}) with $n=2^m$ and for every $k \le d$, with probability greater than $1-\delta/3$. Projecting this estimator onto $\Theta_+$ changes the Frobenius error only by a universal multiplicative constant (arguing as in (\ref{projk}) below), and we denote by $\tilde \theta_{2^m} \in \Theta_+$ the resulting estimator computed from a sample of size $2^m$.

%\smallskip

We now describe the algorithm at the $m$-th step: Split the $2^{m+1}$ observations into two halves and use the first subsample to construct $\tilde \theta_{2^m} \in \Theta_+$ satisfying (\ref{cplan}) with $\pp_\theta$-probability $\ge 1-\delta/3$. Then use the other $2^m$ observations to construct a confidence set $C_{2^m}$ for $\theta$ centred at $\tilde \theta_{2^m}$: if $2^m<d^2$ we take $C_{2^m}$ from (\ref{RSSconf}) and if $2^m \ge d^2$ we take $C_{2^m}$ from (\ref{liconf}) -- in both cases of non-asymptotic coverage at least $1-\alpha, \alpha = \delta/(3T)$ [If $\sigma$ is unknown we proceed as described in Subsection \ref{novar}]. If $|C_{2^m}|_F \le \epsilon$ we terminate the procedure ($m=\hat m$, $\hat n = 2^{\hat m+1}$, $\hat \theta = \tilde \theta_{2^{\hat m}}$), but if $|C_{2^m}|_F>\epsilon$ we repeat the above procedure with $2 \cdot 2^{m+1}=2^{m+1+1}$ new measurements, etc., until the algorithm terminates, in which case we have used $$\sum_{m \le \hat m} 2^{m+1} \lesssim 2^{\hat m} \approx \hat n$$ measurements in total. 

%\medskip

To analyse this algorithm, recall that the quantile constants $z,z_\alpha,\xi_\alpha$ appearing in the confidence sets (\ref{RSSconf}) and (\ref{liconf}) for our choice of $\alpha=\delta/(3T)$ grow at most as $O(\log(1/\alpha))=O(\log T) = o(\overline{\log }d)$. In particular in view of (\ref{cplan}) and (\ref{diameter}) or (\ref{diameter2}) the algorithm necessarily stops at a `maximal sample size' $n=2^{T+1}$ in which the squared Frobenius risk of the maximal model ($k=d$) is controlled at level $\epsilon$. Such $T \in \mathbb N$ is $O(\log (d/\epsilon))$ and depends on $\sigma, d, \epsilon, \delta$, hence can be chosen by the experimenter.

%\smallskip

To prove that this algorithms works we show that the event $$\left\{\|\hat \theta - \theta\|_F^2 > \epsilon^2\right\} \cup \left\{\hat n > \frac{C(\delta)  kd (\log d)^\gamma }{\epsilon^2} \right\} = A_1 \cup A_2$$ has probability at most $2\delta/3$ for large enough $C(\delta), \gamma$. By the union bound it suffices to bound the probability of each event separately by $\delta/3$. For the first: Since $\hat n$ has been selected we know $|C_{\hat n}|_F\le \epsilon$ and since $\hat \theta = \tilde \theta_{\hat n}$ the event $A_1$ can only happen when $\theta \notin C_{\hat n}$. Therefore $$\pp_\theta(A_1) \le \pp_\theta(\theta \notin C_{\hat n}) \le \sum_{m=1}^T \pp_\theta ( \theta \notin C_{2^m}) \le \delta \frac{T}{3T} = \frac{\delta}{3}.$$ For $A_2$, whenever $\theta \in R(k)$ and for all $m\le T$ for which $2^m \ge c'kd \overline{\log } d$, we have, as discussed above, from (\ref{diameter}) or (\ref{diameter2}) and (\ref{cplan}) that $$\ee _\theta |C_{2^m}|_F^2 \le D' \frac{kd \log T}{2^m},$$ where $D'$ is a constant. In the last inequality the expectation is taken under the distribution of the sample used for the construction of $C_{2^m}$, and it holds on the event on which $\tilde \theta_{2^m}$ realises the risk bound (\ref{cplan}). 
Then let $C(\delta), \gamma$ be large enough so that $C(\delta) kd (\log d)^\gamma/ \epsilon^2 \ge c'kd\overline{\log}d$ and let $m_0 \in \mathbb N$ be the smallest integer such that $$2^{m_0} > \frac{C(\delta)  kd (\log d)^\gamma}{\epsilon^2}.$$  Then, for $C(\delta)$ large enough and since $T=O(\log (d/\epsilon)$,
$$\pp_\theta \left(\hat n >\frac{C(\delta)  kd  (\log d)^\gamma}{\epsilon^2}\right) \le \pp_\theta \left(|C_{2^{m_0}}|^2_F >\epsilon^2\right) \le \frac{\ee _\theta|C_{2^{m_0}}|_F^2}{\epsilon^2} \le \frac{D' \log T}{C(\delta) (\log d)^\gamma }<\delta/3,$$  by Markov's inequality, completing the proof.

\begin{remark} [Isotropic sampling] \label{mod} \normalfont The proof above works for isotropic design from Condition \ref{design}a) likewise. When $2^m \ge d^2$ we replace the confidence set (\ref{liconf}) in the above proof by the confidence set from (\ref{Uconf}). Assuming also that $\|\theta\|_F \le M$ for some fixed constant $M$ we can construct a similar upper bound for $T$ and the above proof applies directly (with $T$ of slighter larger but still small enough order). 
\end{remark}

\subsection{Proof of Theorem \ref{RSSthm}}

By Lemma \ref{bernstein} below with $\vartheta=\tilde \theta - \theta$ the $\pp_\theta$-probability of the complement of the event
$$\mathcal E= \left\{\left|\frac{1}{n} \|\mathcal X(\tilde \theta - \theta)\|^2 - \|\tilde \theta- \theta\|_F^2 \right| \le \max\left(\frac{\|\theta-\tilde \theta\|_F^2}{2},  \frac{z d}{n} \right)\right\}$$
is bounded by the deviation terms $2 e^{-c n}$ and $2 e^{-C(K) z}$, respectively (note $z=0$ in Case a)). We restrict to this event in what follows.
We can decompose
\begin{align*}
\hat r_n = \frac{1}{n} \|\mathcal X(\tilde \theta - \theta)\|^2 + \frac{2}{n} \langle \varepsilon, \mathcal X(\theta-\tilde \theta) \rangle + \frac{1}{n} \sum_{i=1}^n (\varepsilon_i^2-\ee \varepsilon_i^2) =A+B+C.
\end{align*}
Since $\pp(Y+Z<0) \le \pp(Y<0) + \pp(Z<0)$ for any random variables $Y,Z$ we can bound the probability
\begin{align*}
\pp_\theta (\theta \notin C_n, \mathcal E) = \pp_\theta \left(\left\{\frac{1}{2}\|\theta-\tilde \theta\|_F^2 > A+B+C + \frac{z d}{n}  + \frac{\bar z +\xi_{\alpha/3, \sigma}}{\sqrt n}\right\}, \mathcal E \right)
\end{align*}
by the sum of the following probabilities
$$I := \pp_\theta\left(\left\{\frac{1}{2}\|\theta-\tilde \theta\|_F^2 > \frac{1}{n} \|\mathcal X(\tilde \theta - \theta)\|^2 + \frac{z d}{n}\right\}, \mathcal E  \right),$$
$$II := \pp_\theta\left(\left\{- \frac{1}{\sqrt n} \langle \varepsilon , \mathcal X(\theta-\tilde \theta)\rangle >   \bar z\right\}, \mathcal E\right), $$
$$III := \pp_\theta \left(- \frac{1}{\sqrt n} \sum_{i=1}^n (\varepsilon_i^2-\ee \varepsilon_i^2) > \xi_{\alpha/3, \sigma}\right). $$
The first probability $I$ is bounded by
\begin{align*}
&\pp_\theta\left(\left\{- \frac{1}{n} \|\mathcal X(\tilde \theta - \theta)\|^2 + \|\theta-\tilde \theta\|_F^2 > \frac{1}{2}\|\theta-\tilde \theta\|_F^2  + \frac{z d}{n}\right\}, \mathcal E  \right) \\
&\le \pp_\theta\left(\left\{\left|\frac{1}{n} \|\mathcal X(\tilde \theta - \theta)\|^2 - \|\tilde \theta- \theta\|_F^2 \right| > \max\left(\frac{\|\theta-\tilde \theta\|_F^2}{2},  \frac{z d}{n} \right)\right\}, \mathcal E\right)=0
\end{align*}
About term $II$: Conditional on $\mathcal X$ the variable $\frac{1}{\sqrt n} \langle \varepsilon , \mathcal X(\theta-\tilde \theta)\rangle$ is centred Gaussian with variance $(\sigma^2/n)\|\mathcal X(\theta-\tilde \theta)\|^2$. The standard Gaussian tail bound then gives by definition of $\bar z$, and conditional on $\mathcal X$,
\begin{align*}
&\le \exp\{-\bar z^2/2(\sigma^2/n)\|\mathcal X(\theta-\tilde \theta)\|^2\} \\
& = \exp\left\{-\frac{z_{\alpha/3} \max(3 \|\theta-\tilde \theta\|^2_F, 4zd/n)}{2\|\mathcal X(\theta-\tilde \theta)\|^2/n}\right\}  \le \exp \{-z_{\alpha/3}\}=\alpha/3
\end{align*}
since, on the event $\mathcal E$,
$$ \max(3\|\theta-\tilde \theta\|^2_F, 4zd/n) \ge (2/n)\|\mathcal X(\theta-\tilde \theta)\|^2.$$ The overall bound for $II$ follows from integrating the last but one inequality over the distribution of $X$. Term $III$ is bounded by $\alpha/3$ by definition of $\xi_{\alpha,\sigma}$.

\begin{remark}[Modification of the proof for Bernoulli errors] \label{bernprf} \normalfont
If instead of Gaussian errors we work with the error model from Subsection \ref{bernoulli}, we require a modified treatment of the terms $II, III$ in the above proof. For the pure noise term $III$ we modify the quantile constants slightly to $\xi_{\alpha, \sigma} = \sqrt{(1/\alpha)}$. If the number $T$ of preparations satisfies $T \ge 4 d^2$ then Chebyshev's inequality and (\ref{bernbds}) give
\begin{align*}
& \pp_\theta \left(\left|\frac{1}{\sqrt n} \sum_{i=1}^n (\varepsilon_i^2-\ee \varepsilon_i^2)\right| > \xi_{\alpha/3, \sigma}\right)  \le \frac{\alpha}{3n} \sum_{i=1}^n\ee  \varepsilon_i^4 \le \frac{\alpha}{3} \frac{4d^2}{T} \le \frac{\alpha}{3}. 
\end{align*}
For the `cross term' we have likewise with $z_{\alpha}=\sqrt {1/\alpha}$ and $a_i = (\mathcal X(\theta-\tilde \theta))_i$ that, on the event $\mathcal E$,
\begin{align*}
\pp_\varepsilon \left(\left\{- \frac{1}{\sqrt n} \langle \varepsilon , \mathcal X(\theta-\tilde \theta)\rangle >   \bar z\right\}, \mathcal E\right)  &\le \frac{1}{n \bar z^2} \ee _\varepsilon \left(\sum_{i=1}^n \varepsilon_i a_i 1_\mathcal E \right)^2\\ 
&\le \frac{d}{T \bar z^2} \frac{\|\mathcal X(\theta-\tilde \theta)\|^2}{n} 1_{\mathcal E}  \le \alpha/3,
\end{align*}
just as at the end of the proof of Theorem \ref{RSSthm}, so that coverage follows from integrating the last inequality w.r.t.~the distribution of $X$. The scaling $T \approx d^2$ is similar to the one discussed in Theorem 3 in ref. \cite{FGLE12}.
\end{remark}

\begin{lemma}\label{bernstein}
a) For isotropic design from Condition \ref{design}a) and any fixed matrix $\vartheta \in \hh_d(\cc)$ we have, for every $n \in \mathbb N$, $$\Pr \left(\left|\frac{1}{n}\|\mathcal X\vartheta\|^2 - \|\vartheta\|_F^2\right| >  \frac{\|\vartheta\|_F^2}{2}\right) \le 2 e^{-c n}.$$  In the standard Gaussian design case we can take $c=1/24$.

b) In the `Pauli basis' case from Condition \ref{design}b) we have for any fixed matrix $\vartheta \in \hh_d(\cc)$ satisfying the Schatten-1-norm bound $\|\vartheta\|_{S_1} \le 2$ and every $n \in \mathbb N$,
$$\Pr \left(\left|\frac{1}{n}\|\mathcal X\vartheta\|^2 - \|\vartheta\|_F^2\right| >  \max\left(\frac{\|\vartheta\|_F^2}{2}, z\frac{d}{n} \right) \right) \le 2  \exp \left\{-C(K) z  \right\}$$ where $C(K) = 1/[(16+8/3)K^2]$, and where $K$ is the coherence constant of the basis.
\end{lemma}
\begin{proof}
We first prove the isotropic case. From (\ref{isoexp}) we see
\begin{align*}
& \Pr \left(\left|\frac{1}{n}\|\mathcal X\vartheta\|^2 - \|\vartheta\|_F^2\right| >  \|\vartheta\|_F^2/2 \right)  = \Pr\left(\left| \sum_{i=1}^n (Z_i^2 - \ee Z^2_1)/\|\vartheta\|_F^2 \right| > n/2 \right)
\end{align*}
where the $Z_i/\|\vartheta\|_F$ are sub-Gaussian random variables. Then the $Z_i^2/\|\vartheta\|_F^2$ are sub-exponential and we can apply Bernstein's inequality (Prop. 4.1.8 in ref. \cite{GN15}) to the last probability. We give the details for the Gaussian case and derive explicit constants. In this case $g_i := Z_i/\|\vartheta\|_F \sim N(0,1)$ so the last probability is bounded, using Theorem 4.1.9 in ref. \cite{GN15}, by
$$\Pr\left(\left| \sum_{i=1}^n (g_i^2 - 1) \right| > \frac{n}{2} \right) \le 2 \exp \left\{- \frac{n^2/4}{4n+2n}\right\},$$ and the result follows.

Under Condition \ref{design}b), if we write $D= \max(n\|\vartheta\|_F^2/2, z d)$ we can reduce likewise to bound the probability in question by
\begin{align*}
& \Pr\left(\left| \sum_{i=1}^n (Y_i - \ee Y_1) \right| > D \right) 
\end{align*}
where the $Y_i = |tr(X^i \vartheta)|^2$ are i.i.d.~bounded random variables. Using $\|E_i\|_{op} \le K/\sqrt d$ from Condition \ref{design}b) and the quantum constraint  $\|\vartheta\|_F \le \|\vartheta\|_{S_1} \le 2$ we can bound
$$|Y_i| \le d^2 \max_i \|E_i\|^2_{op} \|\vartheta\|_{S_1}^2 \le 4K^2d  := U$$ as well as
$$\ee Y_i^2 \le U \ee |Y_i| \le 4 K^2 d \|\vartheta\|_F^2 := s^2.$$  Bernstein's inequality for bounded variables (e.g., Theorem 4.1.7 in ref. \cite{GN15}) applies to give the bound
$$2 \exp\left\{-\frac{D^2}{2ns^2 + \frac{2}{3}UD}\right\} \le 2  \exp \left\{-C(K) z  \right\},$$ after some basic computations, by distinguishing the two regimes of $D=n\|\vartheta\|_F^2/2 \ge zd$ and $D=zd \ge n\|\vartheta\|_F^2/2$.
\end{proof}

\subsection{Proof of Theorem \ref{ustatkill}}
 
Since $\ee _\theta \hat R_n = \|\theta - \tilde \theta\|_F^2$ we have from Chebyshev's inequality
\begin{align*}
\pp_\theta (\theta \notin C_n) &\le \pp_\theta \left( |\hat R_n -\ee \hat R_n| > z_{\alpha, n} \right) \\
& \le \frac{{\rm Var}_\theta(\hat R_n -\ee \hat R_n)}{z_{\alpha_n}^2}.
\end{align*}
Now $U_n =\hat R_n -\ee _\theta \hat R_n$ is a centred U-statistic and has Hoeffding decomposition $U_n = 2L_n + D_n$ where 
$$L_n = \frac{1}{n} \sum_{i =1}^{n} \sum_{m,k} (Y_i X^i_{m,k} - \ee _\theta[Y_i X^i_{m,k}])(\Theta_{m,k}-\tilde \Theta_{m,k})
 $$ is the linear part
 and
 $$D_n = \frac{2}{n(n-1)} \sum_{i<j} \sum_{m,k} (Y_i X^i_{m,k} - \ee _\theta[Y_i  X^i_{m,k}])(Y_j X^i_{m,k} - \ee [Y_j  X^i_{m,k}]) $$
the degenerate part. We note that $L_n$ and $D_n$ are orthogonal in $L^2(\pp_\theta)$.

The linear part can be decomposed into 
$L_n = L_n^{(1)} + L_n^{(2)}$
where 
$$L_n^{(1)} =\frac{1}{n} \sum_{i =1}^{n} \sum_{m,k} \left(\sum_{m',k'} X^i_{m',k'} X^i_{m, k} \Theta_{m',k'} - \Theta_{m,k} \right)(\Theta_{m,k}-\tilde \Theta_{m,k})$$
and 
$$L_n^{(2)}  = \frac{1}{n} \sum_{i =1}^{n} \varepsilon_i \sum_{m,k} X^i_{m,k}(\Theta_{m,k}-\tilde \Theta_{m,k}).$$
Now by the i.i.d.~assumption we have $${\rm Var}_\theta(L_n^{(2)}) = \sigma^2\frac{\|\tilde \theta - \theta\|^2_F}{n}.$$
Moreover, by transposing the indices $m,k$ and $m',k'$ in an arbitrary way into single indices $M=1, \dots, d^2, K=1, \dots, d^2$, $d^2=p$, respectively, basic computations given before eq.~(28) in ref. \cite{NvdG13} imply that the variance of the second term is bounded by $${\rm Var}_\theta(L^{(1)}_n) \leq \frac{c\|\theta - \tilde \theta\|^2_F \|\theta\|^2_F}{n}$$ where $c$ is a constant that depends only on $\ee X_{1,1}^4$ (which is finite since the $X_{1,1}$ are sub-Gaussian in view of Condition \ref{design}a)). Moreover, the degenerate term satisfies $${\rm Var}_\theta(D_n) \leq c \frac{d}{n^2} \|\theta\|_F^4$$ in view of standard $U$-statistic computations leading to eq.\ (6.6) in ref. \cite{ITV10}, with $d^2=p$, and using the same transposition of indices as before. This proves coverage by choosing the constants in the definition of $z_{\alpha,n}$ large enough. 
 
\subsection{Proof of Theorem \ref{main}}

We prove the result for symmetric matrices with real entries -- the case of Hermitian matrices requires only minor (mostly notational) adaptations.

%\medskip

Given the estimator $\tilde \theta_{\rm Pilot}$, we can easily transform it into another estimator $\tilde \theta$ for which the following is true.

\begin{theorem}\label{estcond}
There exists an estimator $\tilde \theta$ that satisfies, uniformly in $\theta \in R(k)$, for any $k \le d$ and with $\pp_\theta$-probability greater than $1-2\delta/3$, 
\begin{align*}%\label{eq:bibi}
\|  \tilde \theta - \theta\|_{F} \leq r_n(k),
\end{align*}
as well as,
$$\tilde \theta \in R(k),$$ and then also
\begin{align*}%\label{eq:bibi}
 \| \tilde \theta- \theta\|_{S_1} \leq \sqrt{2k}r_n(k).
\end{align*}
\end{theorem}
\begin{proof}
Let $\tilde \theta_{\rm Pilot}$ and let $\tilde \theta$ be the element of $R(d)$ with smallest rank $k'$ such that $$\|\tilde \theta_{\rm Pilot} - \tilde \theta \|_F^2 \le \frac{r^2_n(k')}{4}.$$ Such $\tilde \theta$ exists and has rank $\le k$, with probability $\ge 1-2\delta/3$, since $\theta \in R(k)$ satisfies the above inequality in view of (\ref{risk}). The $\|\cdot\|^2_F$-loss of $\tilde \theta$ is no larger than $r_n(k)$ by the triangle inequality $$\|\tilde \theta - \theta\|_F \le \|\tilde \theta - \tilde \theta_{\rm Pilot} \|_F + \|\tilde \theta_{\rm Pilot} - \theta\|_F,$$ and this completes the proof of the third claim in view of (\ref{l1l2}). 
\end{proof}

%\medskip

The rest of the proof consists of three steps: The first establishes some auxiliary empirical process type results, which are then used in the second step to construct a sufficiently good simultaneous estimate of the eigenvalues of $\theta$. In Step III the coverage of the confidence set is established.

%\medskip

\textbf{STEP I}

Let $\theta \in R^+(k) = R(k) \cap \Theta_+$ and let $\tilde \theta$ be the estimator from Theorem \ref{estcond}. Then with probability $\ge 1-2\delta/3$, and if $\eta = \tilde \theta - \theta$, we have 
\begin{equation}
\|\eta\|_F^2 \le r^2_n(k)~~~~\forall \theta \in R^+(k),
\end{equation}
and that $\eta \in R(2k).$ For the rest of the proof we restrict in what follows to the event of probability greater than or equal to $1-2\delta/3$ described by a) and b) in the hypothesis of the theorem. 

Write $Y_i' = Y_i - tr(X^i \tilde \theta)$ for the `new observations' $$Y_i' = tr (X^i \eta) + \varepsilon_i, ~~i=1, \dots, n.$$ For any $d \times d'$ matrix $V$ we set $$\tilde \gamma_\eta(V) = V^T \left(\frac{1}{n} \sum_{i=1}^n X^i Y_i' \right) V$$ which estimates $$\gamma_\eta(V)= V^T\eta V.$$ 
Let now $U$ be any unit vector in $\rr^d$. Then in the above notation ($d'=1$) we can write
\begin{align*}
\tilde \gamma_\eta(U) &= \frac{1}{n} \sum_{i=1}^n \sum_{m, m' \le d} U_m U_{m'} X^{i}_{m,m'}Y_i' \\
&= \frac{1}{n} \sum_{i=1}^n \sum_{m, m' \le d} U_m U_{m'} X^{i}_{m,m'}(tr (X^i\eta) + \varepsilon_i) \\
&= \frac{1}{n} \sum_{i=1}^n \sum_{m, m' \le d} U_m U_{m'} X^{i}_{m,m'}\left(\sum_{k,k'\le d} X^i_{k,k'} \eta_{k,k'}+ \varepsilon_i \right).
\end{align*}
%Let now $\mathbb U$ be the $d^2 \times 1$ vector with entries $\mathbb U_M = u_m u_{m'}$ (the vectorization of $UU^T$), recall the matrix $\mathbb X$ from before Condition \ref{MRIP}, and let $\hh$ be the $d^2 \times 1$ vector obtained from vectorising $\eta$ (column wise). 
If $\mathbb U$ denotes the $d \times d$ matrix $UU^T$, the last quantity can be written as
$$\frac{1}{n} \langle \mathcal X \mathbb U,  \mathcal X\eta \rangle + \frac{1}{n} \langle \mathcal X \mathbb U, \varepsilon \rangle.$$
We can hence bound, for $\mathcal S = \{U \in \rr^d: \|U\|_2=1\}$
\begin{align*}
& \sup_{\eta \in R(2k), \|\eta\|_F \le r_n(k), U \in \mathcal S} |\tilde \gamma_\eta (U)-\gamma_\eta(U)| \\
&\le \sup_{\eta \in R(2k), \|\eta\|_F \le r_n(k), U \in \mathcal S} \left|\frac{1}{n} \langle \mathcal X \mathbb U, \mathcal X \eta \rangle - \langle \mathbb U, \eta \rangle\right|  + \sup_{U \in \mathcal S} \left|\frac{1}{n} \langle \mathcal X \mathbb U, \varepsilon \rangle\right|.
\end{align*}

\begin{lemma}
The right hand side on the last inequality is, with probability greater than $1-\delta$, of order $$v_n  := O \left( r_n(k) \tau_n(k) + \sqrt{\frac{d}{n}}\right).$$
\end{lemma}
\begin{proof}
The first term in the bound corresponds to the first supremum on the right hand side of the last inequality, and follows directly from the matrix RIP (and Lemma \ref{pythagoras}). For the second term we argue conditionally on the values of $\mathcal X$ and on the event for which the matrix RIP is satisfied. We bound the supremum of the Gaussian process $$\mathbb G_\varepsilon(U) := \frac{1}{\sqrt n} \langle \mathcal X \mathbb U, \varepsilon\rangle  \sim N(0, \|\mathcal X \mathbb U\|^2/n)$$ indexed by elements $U$ of the unit sphere $\mathcal S$ of $\rr^d$, which satisfies the metric entropy bound $$\log N(\delta, \mathcal S, \|\cdot\|) \lesssim d \log (A/\delta)$$ by a standard covering argument. Moreover $\mathbb U = UU^T \in R(1)$ and hence for any pair of vectors $U, \bar U \in \mathcal S$ we have that $\mathbb U - \bar {\mathbb U} \in R(2)$. From the RIP we deduce for every fixed $U, \bar U \in \mathcal S$ that 
\begin{align*}
\frac{1}{n}\|\mathcal X \mathbb U- \mathcal X \bar{\mathbb  U}\|^2 &= \|\mathbb U - \bar{\mathbb U}\|_F^2 \left(1 + \frac{\frac{1}{n}\|\mathcal X (\mathbb U- \bar {\mathbb U})\|^2 - \|\mathbb U-\bar {\mathbb U}\|_F^2}{\|\mathbb U-\bar {\mathbb U}\|_F^2} \right) \\
&\le (1+\tau_n(2)) \|\mathbb U-\bar {\mathbb U}\|_F^2 \le C \|U- \bar U\|^2
\end{align*}
since $\tau_n(2) = O(1)$ and since 
\begin{align*}
\| \mathbb U-\bar {\mathbb U}\|_F^2 &= \sum_{m,m'} (U_mU_{m'}- \bar U_m \bar U_{m'})^2\\ 
&=  \sum_{m,m'} (U_mU_{m'}- U_m\bar U_{m'} + U_m \bar U_{m'} - \bar U_m \bar U_{m'})^2  \le 
2 \|U-\bar U\|^2.
\end{align*}
Hence any $\delta$-covering of $\mathcal S$ in $\|\cdot\|$ induces a $\delta/C$ covering of $\mathcal S$ in the intrinsic covariance $d_{\mathbb G_\varepsilon}$ of the (conditional on $\mathcal X$) Gaussian process $\mathbb G_\varepsilon$, i.e.,
$$\log N(\delta, \mathcal S, d_{\mathbb G_\varepsilon}) \lesssim d \log (A'/\delta)$$ with constants independent of $X$. By Dudley's metric entropy bound (e.g., ref.\ \cite{GN15}) applied to the conditional Gaussian process we have for $D>0$ some constant
$$\ee  \sup_{U \in \mathcal S} |\mathbb G_\varepsilon(U)| \lesssim \int_0^D \sqrt {\log N(\delta, \mathcal S, d_{\mathbb G_\varepsilon})} d \delta  \lesssim \sqrt d $$ and hence we deduce that
\begin{equation}
\ee _\varepsilon \sup_{U \in \mathcal S}\frac{1}{n}\left|\langle \mathcal X \mathbb U, \varepsilon\rangle \right| =   \ee _\varepsilon \frac{1}{\sqrt n}\sup_{U \in \mathcal S} |\mathbb G_\varepsilon(U)| \lesssim \sqrt{\frac{d}{n}}
\end{equation}
with constants independent of $X$, so that the result follows from applying Markov's inequality.
\end{proof}

%\medskip

\textbf{STEP II:}

Define the estimator $$\hat \theta' = \tilde \theta + \frac{1}{n} \sum_{i=1}^n X^i Y_i' = \tilde \theta + \tilde \gamma_\eta (I_d).$$ Then we can write, using $U^T \tilde \gamma_\eta(I_d)U = \tilde \gamma_\eta (U)$,
\begin{align*}
U^T \hat \theta' U - U^T \theta U & = U^T(\tilde \theta + \tilde \gamma_\eta(I_d))U - U^T(\tilde \theta + \eta)U \\
& = \tilde \gamma_\eta(U) - \gamma_\eta(U),
\end{align*}
and from the previous lemma we conclude, for any unit vector $U$ that with probability $\ge 1-\delta$,
\begin{equation*}
|U^T \hat \theta' U - U^T \theta U | \le v_n.
\end{equation*}
Let now $\hat \theta$ be any symmetric positive definite matrix such that 
\begin{equation*}
|U^T \hat \theta U - U^T \hat \theta' U | \le v_n.
\end{equation*}
Such a matrix exists, for instance $\theta \in R^+(k)$, and by the triangle inequality we also have
\begin{equation} \label{thetahat}
|U^T \hat \theta U - U^T \theta U | \le 2v_n.
\end{equation}

\begin{lemma}\label{pcaell1}
Let $M$ be a symmetric positive definite $d \times d$ matrix with eigenvalues $\lambda_j$'s ordered such that $\lambda_1 \ge \lambda_2 \ge ... \ge \lambda_d$. For any $j \le d$ consider an arbitrary collection of $j$ orthonormal vectors $\mathcal V_j = (V^\iota: 1 \le \iota \le j)$ in $\rr^d$. Then we have
$$a)~~ \lambda_{j+1} \le \sup_{U \in \mathcal S, U \perp span(\mathcal V_j)} U^TMU,$$ and
$$b)~~~ \sum_{\iota \le j} \lambda_\iota \ge \sum_{\iota \le j} (V^\iota)^T M V^\iota.$$
\end{lemma}

The proof of this lemma is basic and given in the appendix. Let now $\hat R$ be the rotation that diagonalises $\hat \theta$ such that $\hat R^T \hat \theta \hat R = diag(\hat \lambda_j: j =1, \dots, d)$ ordered such that $\hat \lambda_j \ge \hat \lambda _{j+1}$ $\forall j$. Moreover let $R$ be the rotation that does the same for $\theta$ and its eigenvalues $\lambda_j$. We apply the previous lemma with $M= \hat \theta$ and $\mathcal V$ equal to the column vectors $r_\iota: \iota \le l-1$ of $R$ to obtain, for any fixed $l \le j \le d$,
\begin{equation}
\hat \lambda_l \le \sup_{U \in \mathcal S, U \perp span(r_\iota, \iota \le l-1)} U^T \hat \theta U,
\end{equation}
and also that
\begin{equation}
\sum_{l \le j} \hat \lambda_l \ge \sum_{l \le j} r_{l}^T \hat \theta r_{l}.
\end{equation}
From (\ref{thetahat}) we deduce, that
$$\hat \lambda_l \le \sup_{U \in \mathcal S, U \perp span(r_\iota, \iota \le j-1)} U^T \theta U + 2 v_n = \lambda_j +2v_n ~~~\forall ~l \le j,$$
as well as
$$\sum_{l \le j} \hat \lambda_l \ge \sum_{l \le j} r_{l}^T \theta r_{l} - 2jv_n = \sum_{l \le j} \lambda_l - 2jv_n,$$ with probability $\ge 1-\delta$. Combining these bounds we obtain
\begin{equation} \label{evbd}
\left| \sum_{l \le j} \hat \lambda_l - \sum_{l \le j} \lambda_l \right| \le 2jv_n,~~~j \le d.
\end{equation}

%\medskip

\textbf{STEP III}

We show that the confidence sets covers the true parameter on the event of probability $\ge 1-\delta$ on which Steps I and II are valid, and for the constant $C$ chosen large enough.

%\smallskip

Let $\Pi = \Pi_{R^+(2\hat k)}$ be the projection operator onto $R^+(2 \hat k)$. We have $$\|\hat \vartheta - \theta\|_{S_1} \le \|\hat \vartheta - \Pi \theta\|_{S_1} + \|\Pi \theta - \theta\|_{S_1}.$$ We have, using (\ref{evbd}) and Lemma \ref{lem:mati2} below
\begin{align*}
\|\Pi \theta - \theta\|_{S_1} &= \sum_{J > 2\hat k} \lambda_J = 1- \sum_{J \le 2 \hat k} \lambda_J \\
& \le 1- \sum_{J \le 2 \hat k} \hat \lambda_J + 4 \hat k v_n \\
& \le 6 v_n \hat k \le (C/2) \sqrt{\hat k} r_n(\hat k)
\end{align*}
for $C$ large enough.

Moreover, using the oracle inequality (\ref{oracle}) with $S=\Pi\theta$ and (\ref{projk}),
\begin{align*}
\|\hat \vartheta - \Pi \theta\|_{S_1} &\le \sqrt{4 \hat k} \|\hat \vartheta - \Pi \theta\|_{F} \\
& \le \sqrt{4 \hat k} ( \|\hat \vartheta - \theta\|_{F} +  \|\Pi \theta -\theta\|_{F})\\
& \le \sqrt{4 \hat k} (\|\hat \vartheta - \tilde \theta_{\rm Pilot}\|_{F} + \|\tilde \theta_{\rm Pilot} - \theta\|_{F} +  \|\Pi \theta -\theta\|_{F})\\
& \lesssim \sqrt{\hat k} (r_n(\hat k) +  \|\Pi \theta -\theta\|_{F}).
\end{align*}
We finally deal with the approximation error: Note
$$\|\Pi\theta - \theta \|_{F}^2 = \sum_{l > 2 \hat k} \lambda_l^2 \le \max_{l > 2 \hat k} |\lambda_l| \sum_{l>2 \hat k} |\lambda_l|.
$$
By (\ref{evbd}) we know that $$\sum_{l >  \hat k} \lambda_l  = 1- \sum_{l \le \hat k} \lambda_l \le  1- \sum_{l \le \hat k} \hat \lambda_l + 2 v_n \hat k  \le 4 v_n \hat k.$$ Hence out of the $\lambda_l$'s with indices $l >\hat k$ there have to be less than $\hat k$ coefficients which exceed $4 v_n$. Since the eigenvalues are ordered this implies that the $\lambda_l$'s with indices $l>2 \hat k$ are all less than or equal to $4 v_n$, and hence the quantity in the last but one display is bounded by (since $\hat k < 2 \hat k$), using again (\ref{evbd}) and the definition of $\hat k$,
$$4 v_n  \left(1-\sum_{l\le\hat k} |\lambda_l|\right)  \lesssim v_n \left(1-\sum_{l\le\hat k} |\hat \lambda_l| \right) + \hat k v^2_n \lesssim v_n^2 \hat k \lesssim \sqrt {\hat k} r_n(\hat k).$$ Overall we get the bound 
$$\|\hat \vartheta - \Pi \theta\|_{S_1} \lesssim \hat k v_n \lesssim (C/2) \sqrt{\hat k} r_n(\hat k) $$
for $C$ large enough, which completes the proof of coverage of $C_n$ by collecting the above bounds. The diameter bound follows from $\hat k \le k$ (in view of the defining inequalities of $\hat k$ being satisfied, for instance, for $\tilde \theta ' = \theta$, whenever $\theta \in R^+(k_0)$.)

We conclude with the following auxiliary results used above.

\begin{lemma} \label{pythagoras}
Under the RIP (\ref{RIP}) we have for every $1\le k \le d$ that, with probability at least $1- \delta$,
\begin{equation}
\sup_{A, B \in R(k)} \left|\frac{\frac{1}{n}\langle \mathcal XA, \mathcal XB \rangle - \langle A, B \rangle_{F}
 }{\|A\|_F\|B\|_F}\right| \le 10 \tau_n(k).
\end{equation}
\end{lemma}
\begin{proof}
The matrix RIP can be written as
\begin{equation}
	\sup_{A\in R(k)}
		\left| \frac{\langle \mathcal XA, \mathcal XA \rangle}{n \langle A,A\rangle_F} -1 \right|
	=\frac{|\langle  A , (n^{-1} M- \id) A \rangle_F|}{\langle A,A\rangle_F}\leq \tau_n(k),
\end{equation}
for a suitable $M\in \hh_{d^2}(\cc)$. The above bound then follows from applying
the Cauchy-Schwarz inequality to
\begin{equation}
	\frac{1}{n}\langle \mathcal XA, \mathcal XB \rangle - \langle A, B \rangle_F 
	=\langle  A , (n^{-1} M- \id) B \rangle_F .
\end{equation}

%
%The result follows from elementary properties of Hilbert space.

\end{proof}

%\medskip

The proof of the following basic lemma is left to the reader.

%\begin{lemma}\label{lem:mati2}
%Let $M$ be a symmetric and positive semi-definite matrix, with positive eigenvalues $(\lambda_j)_j$ ordered in decreasing order. %Let $R^+(j-1)= R(j-1) \cap \Theta_+$. Then for any $2 \leq j \leq d$ we have
%$$\sum_{j'\geq j} \lambda_{j'} =  \|M - R^+(j-1)\|_{S_1}.$$
%\end{lemma}
%
\begin{lemma}\label{lem:mati2}
Let $M\geq 0$ with positive eigenvalues $(\lambda_j)_j$ ordered in decreasing order. 
Denote with $\Pi_{R^+(j-1)}$ 
the projection onto $R^+(j-1)= R(j-1) \cap \Theta_+$. Then for any $2 \leq j \leq d$ we have
$$\sum_{j'\geq j} \lambda_{j'} =  \|M - \Pi_{R^+(j-1)} M\|_{S_1}.$$
\end{lemma}

\textbf{Acknowledgements.} This work has been supported by the EU (SIQS, RAQUEL), the ERC (TAQ, UQMSI) and the DFG (SPP1798, MuSyAd Emmy Noether grant). AC worked on this project while a postdoc at the University of Cambridge. We also acknowledge discussions with C. Riofrio.

\section{Appendix}

\subsection{Pauli spin measurements \& Quantum Tomography}

This work was partly motivated by a problem arising in present-day
physics experiments that aim at estimating quantum states. Conceptually, a quantum mechanical experiment involves two stages: A \emph{source} 
(or \emph{preparation procedure}) that emits quantum mechanical
systems with unknown properties, and a \emph{measurement device} that
interacts with incoming quantum systems and produces real-valued
measurement outcomes, e.g.\ by pointing a dial to a value on a scale.
Quantum mechanics stipulates that both stages are completely described
by certain matrices. 

\smallskip
The properties of the source are represented by a positive
semi-definite unit trace matrix $\theta$, the \emph{quantum state},
also referred to as 
\emph{density matrix}. 
In turn, the measurement device is modelled by a Hermitian
matrix $X$, which is referred to as an \emph{observable} in physics
jargon. 
A key axiom of the quantum mechanical formalism states that if the
measurement $X$ is repeatedly performed on systems emitted by the
source that is preparing $\theta$,
%in an i.i.d.\ fashion, 
% dg: id' avoid i.i.d. here, as ``distributed'' really only applies to
% classical RVs, which we don't have here
then the real-valued measurement outcomes will fluctuate
randomly with expected value
% dg: not italitc, as we are not defining it. it's the classical
% notion. 
\begin{equation}
	\langle X, \theta\rangle_F= {tr} (X \theta).
\end{equation}
The precise way in which physical properties are represented by these
matrices is immaterial to our discussion (cf.\ any textbook, e.g.\
ref.~\cite{peres}). 
We merely note that, while in principle \emph{any} Hermitian $X$ can
be measured by some physical apparatus, the required experimental
procedures are prohibitively complicated for all but a few highly
structured matrices. This motivates the introduction of \emph{Pauli
designs} below, which correspond to fairly tractable `spin
measurements'.

\smallskip

The \emph{quantum state estimation} or \emph{quantum
tomography}\footnote{
	The term `tomography' goes back to the use of \emph{Radon
	transforms} in early schemes for estimating quantum states of
	electromagnetic fields \cite{leonhardt,Gill}. It has become synonymous
	with `quantum density matrix estimation', even though current
	methods applied to quantum systems with a finite dimension $d$
	have no technical connection to classical tomographic
	reconstruction algorithms.
}
problem is to estimate an unknown density matrix $\theta$ from the
measurement of a collection of observables $X^1, \dots, X^n$. This task is of
particular importance to the young field of quantum information
science \cite{nielsen}. There, the sources might be carefully
engineered components used for technological applications such as
quantum key distribution or quantum computing. In this context,
quantum state estimation is the process of characterising the
components one has built -- clearly an important capability for any
technology. 

\smallskip

A major challenge lies in the fact that relevant instances
are described by $d \times d$-matrices for fairly large dimensions
$d$ ranging from 100 to 10.000 in presently performed experiments \cite{Blatt}. 
Such high-dimensional estimation problems can benefit substantially
from structural properties of the objects to be recovered. 
Fortunately, the density matrices occurring in quantum information
experiments are typically well-approximated by matrices of \emph{low
rank} $r \ll d$. 
In fact, in the practically most important applications, one usually
even aims at preparing a state of unit rank -- a so-called \emph{pure
quantum state}.

\subsubsection{Pauli observables}\label{qmeas}

We now introduce a paradigmatic set of quantum measurements that is
frequently used in both theoretical and practical treatments of
quantum state estimation (see, e.g., refs.\ \cite{GLFBE10,Blatt}).  For a more
general account, we refer to standard textbooks \cite{holevo2001statistical,nielsen}. 
The purpose of this section is to motivate the `Pauli design' case
(Condition \ref{design}b) of the main theorem, as well as the approximate Gaussian noise model described in Subsection \ref{bernoulli}. 

We start by describing `spin measurements' on a single
`spin-$1/2$ particle'. Such a measurement corresponds 
to the situation of having $d=2$.  Without worrying about the physical significance,
we accept as fact that on such particles, one may measure one of
three properties, referred to as the `spin along the $x, y$, or
$z$-axis' of $\rr^3$. 
Each of these measurements may yield one of two outcomes, denoted by
$+1$ and $-1$ respectively.

The mathematical description of these measurements is derived from the
\emph{Pauli matrices}
%The choice of $x_k$ for the individual spin labelled $k=1,\dots, N$ 
%corresponds to the specific Pauli spin matrix
%that is being measured in the experiment, i.e., 
%which of the familiar Hermitian matrices
\begin{equation}
	\sigma^1= \left[
	\begin{array}{cc}
	0 & 1\\
	1 & 0
	\end{array}
	\right],\,
	\sigma^2=\left[
	\begin{array}{cc}
	0 & -i\\
	i & 0
	\end{array}
	\right],\,
	\sigma^3 =\left[
	\begin{array}{cc}
	1 &0\\
	0 & -1
	\end{array}
	\right]
\end{equation}	
in the following way. 
%
%
%from $\mm_2(\cc)$ is taken.
Recall that
the Pauli matrices have eigenvalues $\pm 1$. 
For $x\in\{1,2,3\}$ and
$j\in\{+1, -1\}$,
we write
$\psi_j^x$
for the 
normalised
eigenvector of $\sigma^x$ with eigenvalue $j$.
The spectral decomposition of each Pauli spin matrix can hence
be expressed as
\begin{equation}\label{eqn:spectral}
	\sigma^x = \pi^{x}_+ - \pi^{x}_-,
\end{equation}
with 
\begin{equation}
	\pi^{x}_{\pm} = \psi^{x}_{\pm} (\psi^{y}_{\pm})^\ast\geq 0
\end{equation}
denoting the projectors onto the eigenspaces. 
Now, a physical measurement of
the `spin along direction $x$' 
on a system in state $\theta$
will give rise to a $\{-1, 1\}$-valued
random variable $C^x$ with
\begin{equation}
	\pp(C^x = j) =  tr\left(
		\pi^{x}_j \theta
		%\pi^{x_1}_{j_1} \otimes \dots \otimes  \pi^{x_N}_{j_N}
	\right),
\end{equation}
where $\theta\in \hh_2(\cc)$.
Using eq.~(\ref{eqn:spectral}), this is equivalent to
stating that the expected value of $C^x$ is given by
\begin{equation}\label{eqn:singlepauli}
	\ee(C^x) =  tr\left(
		\sigma^x
		\theta
		%\pi^{x_1}_{j_1} \otimes \dots \otimes  \pi^{x_N}_{j_N}
	\right).
\end{equation}

\smallskip
Next, we consider the case of joint spin measurements on a collection
of $N$ particles. For each, one has to decide on an axis for the spin
measurement. Thus, the joint \emph{measurement setting} is now
described by a word $x=(x_1,\dots, x_N)\in  \{1,2,3\}^N$. 
The axioms of quantum mechanics posit that the
 joint state $\theta$ of the $N$ particles acts on 
the tensor product space $(\cc^2)^{\otimes N}$,
so that $\theta\in \hh_{2^N}(\cc)$.
\smallskip

Likewise,
the \emph{measurement outcome} is a word $j=(j_1, \dots, j_N)\in\{1,
-1\}^N$, with $j_i$ the value of the spin along axis $x_i$ of particle
$i=1,\dots, N$. As above, this prescription gives rise to a $\{1,-1\}^N$-valued
random variable $C^x$. Again, the axioms of quantum mechanics 
imply that the
distribution of $C^x$ is given by
\begin{equation}\label{eqn:multispectral}
	\pp(C^x = j) =  tr\left((\pi^{x_1}_{j_1} \otimes \dots \otimes  \pi^{x_N}_{j_N}) \theta
	\right).
\end{equation}
Note that the components of the random vector $C^x$ are not
necessarily independent, as $\theta$ will generally not factorise
%(this is linked to the oft-quoted phenomenon of \emph{entanglement} \cite{nielsen} which is, however, not relevant to the problems we are interested in here).

\smallskip
It is often convenient to express the 
information in eq.~(\ref{eqn:multispectral}) in a way that involves
tensor products of Pauli matrices, rather than their spectral
projections. In other words, we seek a generalisation of eq.\
(\ref{eqn:singlepauli}) to $N$ particles. As a first step toward this
goal, let 
\begin{eqnarray}
	\chi(j) = 
	\left\{
		\begin{array}{ll}
			-1\qquad&\text{number of $-1$ elements in $j$ is odd} \\
			1\qquad&\text{number of $-1$ elements in $j$ is even}
		\end{array}
	\right.
\end{eqnarray}
be the \emph{parity function}. Then one easily verifies
\begin{equation} 
	tr ((\sigma^{x_1} \otimes \dots \otimes \sigma^{x_N} )\theta)
	=
	\sum_{j\in\{1,-1\}^N} 
	\chi(j)
	\,
%	\pp(C^x=j)
	tr
	\left(
	\theta
		(
		\pi^{x_1}_{j_1} \otimes \dots \otimes  \pi^{x_N}_{j_N}
		)
	\right)
	= \ee\big(\chi(C^x)\big).
\end{equation}
In this sense, the tensor product 
%of Pauli matrices
$\sigma^{x_1}\otimes\dots\otimes\sigma^{x_N}$ describes a
measurement of the parity of the spins along the 
respective directions given by $x$. 

\smallskip
In fact, the entire distribution of $C^x$ can be expressed in terms of
tensor products of Pauli matrices and suitable parity functions. To
this end, we extend the definitions above. Write
\begin{equation}
	\sigma^0=\left[
	\begin{array}{cc}
	1 &0\\
	0 & 1
	\end{array}
	\right]
\end{equation}
for the identity matrix in $\mm_2(\cc)$. For every subset
$S$ of $\{1, \dots, N\}$, define the `parity function restricted to
$S$' via
\begin{eqnarray}
	\chi_S(j) = 
	\Big\{
		\begin{array}{ll}
			-1\qquad&\text{number of $-1$ elements $j_i$ for $i\in S$ is odd} \\
			1\qquad&\text{number of $-1$ elements $j_i$ for $i\in S$ is even}. 
		\end{array}
	%\right
\end{eqnarray}
Lastly, for $S\subset \{1, \dots, N\}$ and $x\in\{1,2,3\}^N$,
the \emph{restriction of $x$ to $S$} is 
\begin{eqnarray}
	x^S_i = 
	\left\{
		\begin{array}{ll}
			x_i\qquad& i \in S \\
			0\qquad& i \not\in S.
		\end{array}
	\right.
%	\subset \{ 0, 1, 2, 3\}^N.
\end{eqnarray}
Then for every such $x, S$ one verifies the identity
\begin{equation}\label{expsam}
	tr ((\sigma^{x^S_1} \otimes \dots \otimes \sigma^{x^S_N})\theta)
	=
	\ee\big( \chi_S(C^x) \big).
\end{equation}
In other words, the distribution of $C^x$ 
contains enough 
information to compute the
expectation value of all observables
$(\sigma^{x^S_1} \otimes \dots \otimes \sigma^{x^S_N})$ that can be
obtained by replacing the Pauli matrices on an  arbitrary subset $S$
of particles by the identity $\sigma^0$.
The converse is also true: the set of all such expectation values
allows one to recover the distribution of $C^x$. The explicit formula
reads
\begin{align}\label{eqn:invft}
	\pp( C^x = j ) 
	&=
	\frac{1}{2^N}\, \sum_{S \subset \{1, \dots, N\}} 
		\chi_S(j)\,
		\ee\big( \chi_S (C^x) \big) \\
	&=
	\frac{1}{2^N}\, \sum_{S \in \{1, \dots, N\}} 
		\chi_S(j)\,
		tr \big(\theta (
			\sigma^{x_1^S} \otimes \dots \otimes \sigma^{x_N^S}
			)
		\big) \notag
\end{align}
and can be verified by direct computation. [Note that 
$\ee\big(\chi_S(C^x)\big)$ is effectively a Fourier coefficient
(over the group $\mathbb{Z}_2^N$) of
the distribution function of the $\{-1,1\}^N$-valued random variable $C^x$.
Equation~(\ref{eqn:invft}) is then nothing but an inverse Fourier
transform.]

\smallskip
In this sense, the information obtainable from joint spin measurements on
$N$ particles can be encoded in the $4^N$ real numbers
\begin{equation}\label{eqn:pauliexpansion}
	2^{-N/2}\,tr ((\sigma^{y_1} \otimes \dots \otimes \sigma^{y_N})\theta),
	\qquad
	y \in \{ 0, 1, 2, 3 \}^N.
\end{equation}
Indeed, every such $y$ arises as $y=x^S$ for some (generally
non-unique) combination of $x$ and $S$.
This representation is particularly convenient from a mathematical
point of view, as the collection of matrices
\begin{equation} \label{pauldef}
	E^y := 2^{-N/2}\sigma^{y_1}\otimes \dots \otimes  \sigma^{y_N},
	\qquad
	y \in \{ 0, 1, 2, 3 \}^N
\end{equation}
forms an ortho-normal basis with respect to the $\langle \cdot, \cdot \rangle_F$ inner
product. Thus the terms in eq.~(\ref{eqn:pauliexpansion}) are just the
coefficients of a basis expansion of the density matrix $\theta$.\footnote{We note that quantum mechanics allows to 
design measurement devices that directly probe the observable 
of $\sigma^{y_1} \otimes\dots\otimes \sigma^{y_N}$, without first
measuring the spin of every particle and then computing a parity
function. In
fact, the ability to perform such correlation measurements is crucial
for \emph{quantum error correction protocols} \cite{nielsen}. For 
practical reasons 
these setups are used less commonly in tomography experiments, though.}

\subsubsection{From (\ref{eqn:pauliexpansion}) to Condition \ref{design}b)}

Following \cite{FGLE12} we use eq.~(\ref{eqn:pauliexpansion}) as our model
for quantum tomographic measurements. Note that the $E^y$ satisfy
Condition \ref{design}b) with coherence constant $K=1$ and $d=2^N$.  In the model (\ref{model}) under Condition \ref{design}b) we wish to
approximate $d \cdot tr(E^y \theta)$ for a fixed observable $E^y$ (we
fix the random values of the $X^i$'s here) and for $d=2^N$.  
If $y=x^S$ for some setting $x$ and subset $S$, then the parity
function $B^y:=\chi_S(C^x)$
has expected value $2^{N/2} \cdot tr(E^y \theta)=\sqrt d \cdot
tr(E^y \theta)$ (see eqs.\ (\ref{expsam}) and (\ref{pauldef})), and
itself is a Bernoulli variable taking values $\{1,-1\}$ with
$$p=\pp(B^y=1)=\frac{1+\sqrt d tr(E^y \theta)}{2}.$$ 
Note that 
$$\sqrt d |tr(E^y \theta)| \le \sqrt d \|E^y\|_{op} \|\theta\|_{S_1} \le 1,$$
so indeed $p \in [0,1]$ and the variance satisfies 
$${\rm Var}B^y = 1 - d \cdot tr(E^y \theta)^2 \le 1.$$ 
This is precisely the error model described in Subsection \ref{bernoulli}.

\subsection{Proof of Lemma \ref{pcaell1}}

{\textbf{a):}} Consider the subspaces $E = span((V^{\iota})_{\iota \leq j})^{\perp}$ and $F = span((e_{\iota})_{\iota \leq j+1})$ of $\rr^d$, where the $e_\iota$'s are the eigenvectors of the $d \times d$ matrix $M$ corresponding to eigenvalues $\lambda_j$. Since $\mathrm{dim}(E) + \mathrm{dim}(F) = (d-j) + j+1 = d+1$, we know that $E \bigcap F$ is not empty and there is a vectorial sub-space of dimension $1$ in the intersection. Take $U  \in E \bigcap F$ such that $\|U\| = 1$. Since $U \in F$, it can be written as
$$U = \sum_{\iota=1}^{j+1} u_{\iota} e_{\iota}$$ for some coefficients $u_\iota$. Since the $e_{\iota}$'s are orthogonal eigenvectors of the symmetric matrix $M$ we necessarily have
$$MU = \sum_{\iota=1}^{j+1} \lambda_{\iota} u_{\iota} e_{\iota},$$
and thus
$$U^T MU = \sum_{\iota=1}^{j+1} \lambda_{\iota} u_{\iota}^2.$$
Since the $\lambda_{\iota}$'s are all non-negative and ordered in decreasing absolute value, one has 
$$U^T MU =   \sum_{\iota=1}^{j+1} \lambda_{\iota} u_{\iota}^2 \geq  \lambda_{j+1}  \sum_{\iota=1}^{j+1} u_{\iota}^2 = \lambda_{j+1} \|U\|^2 = \lambda_{j+1}.$$ Taking the supremum in $U$ yields the result.

{\textbf{b):}} For each $\iota \leq j$, let us write the decomposition of $V^{\iota} $ on the basis of eigenvectors $(e_{l}: l\le d)$ of $M$ as $$V^{\iota} = \sum_{l \leq d} v^{\iota}_l e_{l}.$$
Since the $(e_{l})$ are the eigenvectors of $M$ we have
$$\sum_{\iota \leq j} (V^{\iota})^T M V^{\iota} =  \sum_{\iota \leq j}  \sum_{l=1}^d \lambda_{l} (v^{\iota}_l)^2,$$
where $\sum_{l=1}^d (v^{\iota}_l)^2 = 1$ and $\sum_{\iota \leq j} (v^{\iota}_l)^2 \leq 1$, since the $V^\iota$ are orthonormal. The last expression is maximised in $(v^{\iota}_l)_{\iota \leq j, 1 \le l \leq d}$ and under these constraints, when $v^{\iota}_{\iota} = 1$ and $v^{\iota}_{l} = 0$ if $\iota \neq l$ (since the $(\lambda_{\iota})$ are in decreasing order), and this gives
$$\sum_{\iota \leq j} (V^{\iota})^T M V^{\iota}  \leq \sum_{\iota \leq j} \lambda_{\iota}.$$

\end{document}